\documentclass[preprint]{imsart}

\RequirePackage{amsthm,amsmath,amssymb}
\RequirePackage[numbers]{natbib}
\RequirePackage[colorlinks,citecolor=blue,urlcolor=blue]{hyperref}
\RequirePackage{graphicx}

\usepackage[margin=90pt]{geometry}

\startlocaldefs
\numberwithin{equation}{section}
\theoremstyle{plain}
\newtheorem{thm}{Theorem}[section]
\newtheorem{deff}{Definition}[section]
\newtheorem{lem}{Lemma}[section]
\newtheorem{pro}{Proposition}[section]
\newtheorem{rem}{Remark}[section]
\newtheorem{cor}{Corollary}[section]
\newtheorem{con}{Condition}[section]

\endlocaldefs

\newcommand\independent{\protect\mathpalette{\protect\independenT}{\perp}}
\def\independenT#1#2{\mathrel{\rlap{$#1#2$}\mkern2mu{#1#2}}}

\begin{document}

\begin{frontmatter}
\title{Consistent Model Selection of Discrete Bayesian Networks from Incomplete Data}
\runtitle{Model Selection of Discrete Bayesian Networks}

\begin{aug}
\author{\fnms{Nikolay H.} \snm{Balov}\ead[label=e1]{nikolay\_balov@urmc.rochester.edu}}

\address{Department of Biostatistics and Computational Biology \\
University of Rochester Medical Center, Rochester, NY-14642 \\
\printead{e1}}

\runauthor{N.H. Balov}
\affiliation{University of Rochester}

\end{aug}

\begin{abstract}
A maximum likelihood based model selection of discrete Bayesian networks is considered. The structure learning is performed by employing a scoring function $S$, which, for a given network $G$ and $n$-sample $D_n$, is defined as the maximum marginal log-likelihood $l$ minus a penalization term $\lambda_n h$ proportional to network complexity $h(G)$,
$$
S(G|D_n) = l(G|D_n) - \lambda_n h(G).
$$
An available case analysis is developed with the standard log-likelihood replaced by the sum of sample average node log-likelihoods. The approach utilizes partially missing data records and allows for comparison of models fitted to different samples. 

In missing completely at random settings the estimation is shown to be consistent if and only if the sequence $\lambda_n$ converges to zero at a slower than $n^{-{1/2}}$ rate. In particular, the BIC model selection ($\lambda_n=0.5\log(n)/n$) applied to the node-average log-likelihood is shown to be inconsistent in general. This is in contrast to the complete data case when BIC is known to be consistent. The conclusions are confirmed by numerical experiments. 
\end{abstract}

\begin{keyword}[class=AMS]
\kwd[Primary ]{62F12}
\kwd[; secondary ]{62H12}
\end{keyword}

\begin{keyword}
\kwd{Bayesian networks}
\kwd{categorical data}
\kwd{model selection}
\kwd{penalized maximum likelihood}
\kwd{missing completely at random}
\end{keyword}

\tableofcontents

\end{frontmatter}

%%%%%%%%%%%%%%%%%%%%%%%%%%%%%%%%%%%%%%%%%%%%%%%%%%%%%%%%%%%%%%%%%%%%%%%%%%%%%%%%%%%%%%%%%%%%%%%%%%%%%%%%%%%%%%%%%%%%%%%%%%%%%
\section{Introduction} 

The continuing interest in developing sparse statistical models, with the notable presence of Bayesian networks among them, is well motivated by a number of pressing practical problems coming from gene/protein expression analysis and medical imaging, to mention a few. Although graphical probability models based on directed connections between random variables provide efficient joint distribution description, the application of such models is often limited by the ambiguity of their observed behavior which makes the learning rather difficult. 

One of the prevailing approaches to graphical model selection is through optimization of some scoring functions. In the context of Bayesian networks, the usual choice is the log of posterior. Let $(G,\theta)$ be a Bayesian network with graph structure $G$ and probability model parameter $\theta\in\Theta$. Following the Bayesian paradigm (see for example \cite{cooper} and \cite{spiegelhalter}), one specifies prior probability distributions $\pi$ for $G$ and $\theta$. Then, for a sample $D_n$ of size $n$, one considers the Bayesian scoring function  
$$
S(G|D_n) = \log\\ \pi(G) + \log\\ \mathcal{L}(G|D_n),
$$
where 
$$
\mathcal{L}(G|D_n) = \int_{\theta\in\Theta} \mathcal{L}(G,\theta|D_n) \pi(\theta)d\theta
$$
is the so-called marginal likelihood of $G$, while $\mathcal{L}(G,\theta|D_n)$ is the usual likelihood of $(G,\theta)$. 
The Bayesian scoring function measures the posterior certainty under the chosen prior system and the model with maximum score is thus a natural estimator. 

The main virtue of the Bayesian approach is in counter-balancing the tendency of the maximum likelihood estimation to choose the most complex model fitting the data. 
As first noticed by \cite{schwartz}, when the probability parameter space $\Theta$ constitutes an exponential family in an Euclidean space, the marginal log-likelihood of a model $M$ admits the approximation 
$$
\log \mathcal{L}(M|D_n) = \textrm{BIC}(M|D_n) + O_p(1),
$$
based on the so-called Bayesian Information Criterion (BIC), 
$$
\textrm{BIC}(M|D_n) \equiv \log\\ \mathcal{L}(M,\hat\theta_M|D_n) - 0.5 \log(n) \textrm{ dim}(M), 
$$
where $\hat\theta_M$ is the value of $\theta$ that maximizes the log-likelihood for given $M$ and $D_n$, and $\textrm{ dim}(M)$ is the dimension of $M$. 
The immediate application of this result to discrete and conditional Gaussian Bayesian networks was postponed because of the non-Euclidean structure of the parameter space for these models. This obstacle was later overcome by \cite{haughton}, who showed the validity of the BIC approximation for a much large family of curved exponential distributions.  

In a later work, \cite{geiger} applied this result to several families of Bayesian network models including the discrete ones, thus showing the asymptotic consistency of BIC. In its generality, the parameter space $\Theta_G$ of a discrete Bayesian network $G$ comprises a collection of multinomial distributions and the total number of parameters needed to specify them is what is understood as dimension of $\Theta_G$. The BIC approximation is then expressed as 
\begin{equation}
\label{eq:BIC}
\log \mathcal{L}(G|D_n) = \log\\ \mathcal{L}(G,\hat\theta_G|D_n) - 0.5 \log(n) \textrm{ dim}(\Theta_G) + O_p(1).
\end{equation}

Equation \eqref{eq:BIC} suggests a more direct estimating procedure - selecting a model $G$ in $\mathcal{G}$ with maximal BIC score. There are two typical arguments in favor of this route versus the Bayesian one. The first one is methodological - prior based inference is not universally accepted. The other one is computational - calculating marginal likelihoods can be prohibitive, especially so in the framework of large dimensional graphical models. 

These observations have motivated us to pursue the latter, non-Bayesian approach - maximum likelihood estimation  followed by model selection according to some scoring criteria. To generalize it, we reformulate the right-hand side of \eqref{eq:BIC} and consider the following estimation problem 
\begin{equation}
\label{eq:GBIC}
\hat G = arg\max_{G\in\mathcal{G}} \{ n^{-1} \log\\ \mathcal{L}(G,\hat\theta_G|D_n) - \lambda_n h(G) \}, 
\end{equation}
where $\lambda_n$ is some positive sequence and $h$ is a function measuring the complexity of $G$. The class of problems \eqref{eq:GBIC} is known as extended (or penalized) likelihood approach \cite{buntine}. Typical penalization parameters are $\lambda_n=0.5n^{-1}\log(n)$ (BIC) and $\lambda_n=n^{-1}$ (AIC), while $dim(\Theta_G)$ is a usual choice for $h$. We briefly remark that, in order to be useful in practice, the estimation problem \eqref{eq:GBIC} relies on two assumptions: (1) for a fixed $G$, the MLE $\hat\theta_G$ can be easily found, and (2), the set of networks $\mathcal{G}$ is not prohibitively large, which usually requires imposing some network structure restrictions. In this paper however, we are mainly concerned with the theoretical aspects of \eqref{eq:GBIC} - to our knowledge, the consistency properties of $\hat G$ are not investigated in presence of missing values - and present results which are relevant to all estimation algorithms involving penalized log-likelihood of this form.

The paper contributes in three main directions. First, in order to more efficiently handle data with incomplete records, we modify the scoring based model selection \eqref{eq:GBIC} by replacing the log-likelihood function with what we tentatively call node-average log-likelihood (NAL) - a sum of sample average node log-likelihoods relative to the node parents. 
The NAL statistics utilizes partially incomplete sample records instead of discarding them and provides means for comparing models fitted to different samples. We argue that when the number of nodes is large in comparison to the parent sizes, the NAL-based estimation achieves efficiency close to that of the computationally more demanding Expectation Maximization (EM) procedure \cite{lauritzenEM}. 
Second, we focus on missing completely at random data models for they essentially guarantee network identifiability. More general missing at random mechanisms, in most cases, obscure the underlying network structure and render the network unidentifiable. 
Third, we generalize the scoring criteria by allowing the complexity measure $h$ to be any positive function (as long as it is increasing for $G$ as defined later) and a continuum of penalization parameters $\lambda_n=O(n^{-\alpha})$ by specifying a range of possible values $\alpha$ for which the estimation is consistent. 

In Section \ref{ch:formulation} we introduce the notion of node-average log-likelihood and describe the model selection problem in the context of Bayesian networks. Then, Section \ref{ch:model_selection}, we consider the question of network identifiability and formulate consistency in terms of scoring criteria. For the latter we follow \cite{haughton} and \cite{chickering}. 
We show in Section \ref{sec:ident} that if the data is missing completely at random, the identifiability arises under some natural conditions. 
Section \ref{ch:consistent_scores} presents the main result in this paper, Theorem \ref{th:consistency}, claiming that the estimation is asymptotically consistent provided that $\lambda_n$ goes to zero at slower rate than $n^{-1}$, in the complete data case, and $n^{-1/2}$, in presence of missing data. We also show the necessity of the later in missing completely at random settings. Thus, the inconsistency of AIC is (re)confirmed along with somewhat unexpected conclusion regarding the BIC criteria - in the context of NAL optimization, BIC is consistent when applied to complete data but inconsistent otherwise. In Section \ref{ch:experiments} we present some numerical results in confirmation of the theory which are carried out with the {\it catnet} package for {\bf R}. We conclude with a short discussion on possible extensions of the presented approach beyond the class of discrete Bayesian networks. 

%%%%%%%%%%%%%%%%%%%%%%%%%%%%%%%%%%%%%%%%%%%%%%%%%%%%%%%%%%%%%%%%%%%%%%%%%%%%%%%%%%%%%%%%%%%%%%%%%%%%%%%%%%%%%%%%%%%%%%%%%%%%%
\section{Problem formulation and motivation}\label{ch:formulation}

\subsection{Basic definitions}

Let ${\bf X}=(X_i)_{i=1}^N$ be a $N$-vector of discrete random variables. 
Any directed acyclic graph (DAG) $G$ with nodes ${\bf X}$ is a collection of directed edges from parent to child nodes such that there are no cycles. We denote with $Pa_i$ the parents of node $X_i$ in $G$; then $G$ is completely described by the parent sets $\{Pa_i\}_{i=1}^N$. 
The set of all DAGs with nodes ${\bf X}$ admits partial ordering. 
We say that $G_1$ is included in $G_2$ and write $G_1\subseteq G_2$ if all directed edges of $G_1$ are present in $G_2$ as well.
An element $G$ of a set of DAGs $\mathcal{G}$ is called minimal if there is no $\tilde G\in\mathcal{G}$ such that $\tilde G \subset G$; similarly defined are maximal DAGs. In a set of nested DAGs, the minimum and maximum DAG are always uniquely defined. 

Discrete Bayesian network (DBN) on ${\bf X}$ is any pair $(G, P)$ consisting of DAG $G$ and probability distribution $P$ on ${\bf X}$ subject to two conditions:

(1) the joint distribution of {\bf X} given by $P$ satisfies the so-called local Markov property (LMP) with respect to $G$ - any node-variable is independent of its non-descendants given its parents,

(2) $G$ is a minimal DAG compatible with $P$, that is, there is no $\tilde G\subset G$ such that $P$ satisfies LMP with respect to $\tilde G$. 

For any DAG $G$, there is an order of its nodes, called causality order, such that the parents of each one appear earlier in that order. We say that $G$ is compatible with an order $\Omega$ if $Pa_{\Omega(1)}=\emptyset$ and for all $i=2,...,N$, $Pa_{\Omega(i)} \subset \{X_{\Omega(1)},...,X_{\Omega(i-1)}\}$. For $i<j$, we denote with $X_{\Omega(i)}\prec X_{\Omega(j)}$ the fact that $X_{\Omega(i)}$ appears before $X_{\Omega(j)}$ in the order $\Omega$. 

In its generality, the discreteness of our model implies that for each state $x_{Pa_i}$ of the parents of $X_i$, the probability distribution of $X_i$ conditional on $Pa_i=x_{Pa_i}$ is multinomial. Moreover, the conditional probability tables $\{P(X_i|Pa_i)\}_{i=1}^N$ fully specify the joint distribution of {\bf X}. 
Indeed, let $G$ be compatible with an order $\Omega$, that is, $X_{\Omega(1)}\prec X_{\Omega(2)} \prec ... \prec X_{\Omega(N)}$. 
Then, taking into account the LMP, it is evident that with respect to $G$, the joint probability distribution permits the factorization 
$$
P({\bf X}) = \prod_{i=1}^N P(X_{\Omega(i)} | X_{\Omega(1)},...,X_{\Omega(i-1)}) = \prod_{i=1}^N P(X_{\Omega(i)} | Pa_{\Omega(i)}) = \prod_{i=1}^N P(X_i | Pa_i). 
$$
Depending on the context, in a pair $(G, P)$, we shall refer to $P$ either as a joint distribution, such as in the left-hand side of the above display, or as a set of conditional probability tables, as in the right-hand side above. 

For any DAG $G$, there is a maximal set $\mathcal{I}(G)$ of (structural) conditional independence relations of the form $(A \independent B | C)$, for $A,B,C\subset {\bf X}$ and $A,B\ne\emptyset$, determined by LMP \cite{pearl}. On the other hand, $P$ also defines a set of (distributional) independence constraints on ${\bf X}$. Condition (2) in the definition of DBN is needed for assuring that the sets of structural and distributional independence statements in fact coincide. We say that two DAGs $G_1$ and $G_2$ are equivalent and write $G_1\cong G_2$ if $\mathcal{I}(G_1) = \mathcal{I}(G_2)$. With $[G]$ we shall denote the class of DAGs equivalent to $G$. Necessary and sufficient conditions for DAG equivalence can be found in \cite{verma, chickering}. 
We call two DBNs $(G_1, P_1)$ and $(G_2, P_2)$ equivalent if their joint distributions are equal, $P_1=P_2$, which implies equivalence between their graph structures, $G_1\cong G_2$. The essential problem of BN learning is the recovery of the equivalence class $[G]$ from data. 

Another useful notion is that of network complexity. The complexity of a DBN $G$ is typically measured by the number of parameters $df(G)$ needed to specify the conditional probability table of $G$. Let $q(X_i)$ be the number of states, or discrete levels, of $X_i$ and $q(Pa_i) = \prod_{X\in Pa_i} q(X)$ be the number of states of the parent set $Pa_i$. Since for every state of $Pa_i$, $q(X_i)-1$ parameters are needed to define the corresponding multinomial distribution for $X_i$, we have $df(G) = \sum_{i=1}^N q(Pa_i)(q(X_i)-1)$. 

Next we formulate the maximum likelihood estimation (MLE) in the context of DBNs. 
Let $D_n=\{x^{s}\}_{s=1}^{n}$ be a sample of $n$ independent observations on the vector ${\bf X}$. 
Then, the log-likelihood of a DBN $(G,P)$ with respect to $D_n$ is 
\begin{equation}\label{eq:likelihood}
\log \mathcal{L} (G,P|D_n) = \sum_{i=1}^N \sum_{s=1}^n \log P(X_i=x_{i}^s|Pa_i=x_{Pa_i}^s), 
\end{equation}
where $x_i^s$ and $x_{Pa_i}^s$ are the states of $X_i$ and its parent set $Pa_i$ in the $s$-th record $x^s$. According to the ML principle, a DBN estimator can be obtained by maximizing \eqref{eq:likelihood}. Before proceeding with the inference in presence of missing values we need to introduce some useful statistics and convenient notations. 

We write $k\in X_i$ to index the states of $X_i$ and adopt a multi-index notation, $j\in Pa_i$, for the parent configurations of $X_i$. 
Let $1_{i,kj}$ be the indicator function of the event $(X_i = k, Pa_i = j)$. 
For a given sample $D_n$ let us define the counts $n_{i,kj} \equiv \sum_{s=1}^n 1_{i,kj}(x^s)$, $n_{i,j} \equiv \sum_{k\in X_i} n_{i,kj}$ and $n_i \equiv \sum_{j\in Pa_i} n_{i,j}$. 
A record $x^s$ in $D_n$ we shall call incomplete if some of the values $x_i^s$ are missing. 
By convention, if the value of $X_i$ in $x^s$ is missing, then $1_{i,kj}(x^s) = 0$, while if some of the parents in $Pa_i$ are missing, then both $1_{i,kj}(x^s) = 0$ and $1_{i,j}(x^s) = 0$. It is always the case then that $n_i \le n$. 
We shall consider an inference framework using the counts $n_{i,j}$ and $n_{i,kj}$ as statistics summarizing the information in the sample $D_n$.  

Let ${\bf Z} = (Z_i)_{i=1}^N$ be a binary random vector such that $Z_i=1$ if $X_i$ is observed and $Z_i=0$ if it is missing. For an index set $A$ we define $Z_A = \prod_{i\in A} Z_i$. The joint distribution of $({\bf X}, {\bf Z})$ describes all incomplete samples $D_n$ of observations on ${\bf X}$. 

Let us introduce the probabilities $\theta_i$, $\theta_{i,j}$ and $\theta_{i,kj}$ as 
$$
\theta_i \equiv P(Z_i=1, Z_{Pa_i}=1) 
$$
$$
\theta_{i,j} \equiv P(Pa_i=j | Z_i = 1, Z_{Pa_i}=1),
$$
\begin{equation}\label{def:thetas}
\theta_{i,kj} \equiv P(X_i=k | Pa_i=j, Z_i = 1, Z_{Pa_i}=1). 
\end{equation}
With $\theta$ we shall denote the set $\{\theta_i, \theta_{i,j}, \theta_{i,kj}\}_{i,k,j}$ and call it observed conditional probability table of $G$. 

For a sample of fixed size $n$, the random variables $n_i$ and the random vectors $\{n_{i,j}\}_{j\in Pa_i}$ and $\{n_{i,kj}\}_{k\in X_i}$ then satisfy 
$$
n_i|n \sim Binom(\theta_i, n)
$$
$$
\{n_{i,j}\}_j|n_i \sim Multinom(\{\theta_{i,j}\}_j, n_i)
$$
\begin{equation}\label{eq:thetas}
\{n_{i,kj}\}_k|n_{i,j} \sim Multinom(\{\theta_{i,kj}\}_k, n_{i,j}).
\end{equation}
Therefore, as long as $n_i$, $n_{i,j}$ and $n_{i,kj}$ are of interest, the table $\theta$ is all we need to know about the DBN and the mechanism of missingness. 

The usual point estimators of $\theta_i$, $\theta_{i,kj}$ and $\theta_{i,j}$ are 
$$
\hat{\theta}_{i} = \frac{n_{i}}{n} \textrm{, }
\hat{\theta}_{i,j} = \frac{n_{i,j}}{n_i} \textrm{, }
\hat{\theta}_{i,kj} = \frac{n_{i,kj}}{n_{i,j}} .
$$
We shall denote the conditional table defined by $\hat{\theta}$'s with $\hat{\theta}(G| D_n)$ to emphasize that it is estimated for the DAG $G$ from the sample $D_n$. 
The statistics $\hat{\theta}_{i,j}$ and $\hat{\theta}_{i,kj}$ are unbiased estimators of $\theta_{i,j}$ and $\theta_{i,kj}$, respectively 
\begin{equation}\label{eq:thetas_exp}
E \hat{\theta}_{i,j} = E_{n_{i}} E(\frac{n_{i,j}}{n_{i}}|n_{i}) = \theta_{i,j} \textrm{, } E \hat{\theta}_{i,kj} = E_{n_{i,j}} E(\frac{n_{i,kj}}{n_{i,j}}|n_{i,j}) = \theta_{i,kj}.
\end{equation}

The missing data distribution usually belongs to one of the following categories:  

(i) The data is missing completely at random (MCAR) when the missing probabilities are unrelated to either the observed or the unobserved values. In this case ${\bf Z}$ is independent of ${\bf X}$ and we have $\theta_{i,j} = P(Pa_i=j)$ and $\theta_{i,kj} = P(X_i=k | Pa_i=j)$. 

(ii) The data is missing at random (MAR) when the missing probabilities depend on the observed values but not on the unobserved ones. Let us consider a special case of MAR when for each $i$, there is $C_i\subset {\bf X}$ such that $X_i\notin C_i$, $C_i \cap Pa_i = \emptyset$  and $(Z_i, Z_{Pa_i})$ is independent of $(X_i,Pa_i)$ given $C_i$. If furthermore $C_i$ has no descendants of $X_i$, then, by application of LMP, 
$\theta_{i,kj} = P(X_i=k | Pa_i=j)$ holds. For a general MAR however the latter may not be true. 

(iii) If the missing probabilities depend on the unobserved values we have not missing at random (NMAR) case and then neither $\theta_{i,j} = P(Pa_i=j)$ nor $\theta_{i,kj} = P(X_i=k | Pa_i=j)$ hold anymore. 

As we discuss in Section \ref{sec:ident}, the missing data distribution is implicated in network identifiability. In this regard, the MCAR model is the most transparent one for it does not interfere with the network topology. 

%%%%%%%%%%%%%%%%%%%%%%%%%%%%%%%%%%%%%%%%%%%%%%%%%%%%%%%%%%%%%%%%%%%%%%%%%%%%%%%%%%%%%%%%%%%%%%%%%%%%%%%%%%%
\subsection{Node-average log-likelihood}

We consider two objective functions for estimating DBNs based on the log-likelihood \eqref{eq:likelihood}. The first one is the sample average log-likelihood 
$$
\tilde{l} (G|D_n) = \frac{1}{n} \sum_{i=1}^N \sum_{j\in Pa_i} \sum_{k\in X_i} n_{i,kj} \log \hat{\theta}_{i,kj}
$$
\begin{equation}
\label{eq:totalloglik}
= \sum_{i=1}^N \sum_{j\in Pa_i} \frac{n_{i,j}}{n} \sum_{k\in X_i} \hat{\theta}_{i,kj} \log \hat{\theta}_{i,kj}. 
\end{equation}
When the data has no missing values we have $n\tilde{l} (G|D_n) = \max_{\theta} \log \mathcal{L} (G,\theta|D_n)$. 

The second objective function is the sum of sample average node log-likelihoods
$$
l (G|D_n) 
 = \sum_{i=1}^N \frac{1}{n_{i}} \sum_{j\in Pa_i} \sum_{k\in X_i} n_{i,kj} \log \hat{\theta}_{i,kj}
$$
\begin{equation}
\label{eq:avrloglik}
= \sum_{i=1}^N \sum_{j\in Pa_i} \hat{\theta}_{i,j} \sum_{k\in X_i} \hat{\theta}_{i,kj} \log \hat{\theta}_{i,kj} 
= \sum_{i=1}^N l(X_i|Pa_i,D_n), 
\end{equation}
where $l(X_i|Pa_i,D_n) \equiv \sum_{j\in Pa_i} \hat{\theta}_{i,j} \sum_{k\in X_i} \hat{\theta}_{i,kj} \log \hat{\theta}_{i,kj}$ is known as negative conditional entropy of node $X_i$. 
Hereafter, we drop the qualifier `sample average' from \eqref{eq:totalloglik} and \eqref{eq:avrloglik} and call \eqref{eq:avrloglik} node-average log-likelihood (NAL). 

If $D_n$ is a complete sample, then for every $i$, $n_i = \sum_{j\in Pa_i} n_{i,j} = n$. Hence $\hat\theta_{i,j}\hat\theta_{i,kj} = n_{i,kj}/n$ and consequently $\tilde{l}(G|D_n) = l(G|D_n)$. 
If the data is incomplete however, we may have $n_i < n$ and then \eqref{eq:totalloglik} and \eqref{eq:avrloglik} will be different. 
In the latter case, the log-likelihood \eqref{eq:totalloglik} may have imbalanced representation of the potential parent sets. For example, if for two different parent sets $Pa_i$ and $Pa_i'$ of the i-th node $n_i(Pa_i') < n_i(Pa_i)$, then $Pa_i'$ might be preferably selected due to the smaller size of the subsample that represents it in \eqref{eq:totalloglik} even when $Pa_i'$ has worse fit than $Pa_i$, i.e. $l(X_i|Pa_i',D_n) < l(X_i|Pa_i,D_n)$. 
The simplest solution to this problem - discarding all incomplete records in the sample - may drastically reduce the effective sample size. On the other hand, \eqref{eq:avrloglik} can utilize all $n_i$ sample records for estimation of $\theta_{i,kj}$'s. Essentially, NAL exploits the decomposable nature of the log-likelihood \eqref{eq:totalloglik} and, by adjusting for the sample size, allows comparison of models fitted to different samples. We mention that, similarly, NAL can be adopted in other decomposable log-likelihood based models. 

It can be easily demonstrated that the maximum likelihood principle alone is inefficient for estimating DBNs. Let us assume for simplicity that $\mathcal{G}$ comprises all DBNs with node order compatible with the index order, $X_i\prec X_2 \prec ...\prec X_N$. The maximum NAL equation, $\hat G = arg\max l(G|D_n)$, will then result in the following estimates for the parents set $Pa_i$ 
$$
\hat{P}a_i = arg\max_{Pa_i \subseteq \{1,...,i-1\}} \hat{\theta}_{i,j} \sum_{k\in X_i} \hat{\theta}_{i,kj} \log \hat{\theta}_{i,kj}.
$$
From the increasing property of the conditional log-likelihood (see Lemma \ref{lemma:lincrease} below) it follows that the solution of the above equation is $Pa_{i}= \{1,...,i-1\}$, for every $i>1$. Thus, the MLE solution will be the most complex DBN in $\mathcal{G}$ and will overestimate the true $G$. %Such is the origin of the Bayesian network selection problem. 
In the remainder of this paper we shall investigate more closely the properties of NAL-based estimation in a model selection context and shall provide criteria for asymptotically consistent estimation. 

%%%%%%%%%%%%%%%%%%%%%%%%%%%%%%%%%%%%%%%%%%%%%%%%%%%%%%%%%%%%%%%%%%%%%%%%%%%%%%%%%%%%%%%%%%%%%%%%%%%%%%%%%%%%%%%%%%%%%%%%%%%%%
\subsection{Relation between NAL maximization and EM algorithm}
\label{ch:em}

In missing data settings, the standard way to utilize all of the available data is to apply an EM algorithm - see \cite{lauritzenEM} for application of EM to Bayesian networks.
For a sample $D_n$ let $D_n^{obs}$ be the observed part of the data. The EM algorithm involves the following conditional expectation 
$$
Q(G,P|G',P') \equiv E(\log P(D_n|G,P) | D_n^{obs}, G', P')
$$
$$
= \sum_{i=1}^N \sum_{j\in Pa_i, k\in X_i} E( \sum_{s=1}^n 1_{\{x_i^s =k, x_{Pa_i}^s = j\}} | D_n^{obs}, G', P')\log P_{i,kj},
$$
where $P_{i,kj} = P(X_i=k|Pa_i=j)$. 
Finding $Q$ implements the E-step of the algorithm. The M-step maximizes $Q(G,P|G',P')$ for $G$ and $P$. Solutions of the EM algorithm are all $(\hat G, \hat P)$ such that 
$Q(\hat G, \hat P|\hat G, \hat P) = \max_{G, P}Q(G, P|\hat G, \hat P)$. 

It can be shown that NAL maximization is equivalent to solving a sub-optimal EM algorithm with $\sum_{s=1}^n 1_{\{x_i^s =k, x_{Pa_i}^s = j\}}$ replaced by the sum $n_{i,kj} + n_{i,kj}^{mis}$, where $n_{i,kj}$ and $n_{i,kj}^{mis}$ are the number of records in $D_n$ for which the event $(X_i=k,Pa_i=j)$ is observed and missing, respectively. For each $i$ , this is equivalent to replacing $D_{n}^{obs}$ by a sub-sample $D_{n,i}^{obs}$ with all $x^s$ from $D_n^{obs}$ for which $(X_i,Pa_i)$ is not fully observed being removed. Let $n_i^{mis} = \sum_{k,j} n_{i,kj}^{mis} = n-n_i$. 
Given $D_n^{obs}$, $n_{i,kj}$, $n_i$ and $n_i^{mis}$ are fixed but $n_{i,kj}^{mis}$ is random. In fact, $n_{i,kj}^{mis}$ conditional on $(n_i=n_i^{mis}, G', P')$ follows a Binomial distribution. 
Since $E( \sum_{s=1}^n 1_{\{x_i^s =k, x_{Pa_i}^s = j\}} | D_{n,i}^{obs}, G', P') = n_{i,kj}+E(n_{i,kj}^{mis} | n_i=n_i^{mis}, G', P')$, we define 
$$
Q^*(G,P|G',P') \equiv \sum_{i=1}^N \sum_{j\in Pa_i, k\in X_i} (n_{i,kj} + E(n_{i,kj}^{mis} | n_i=n_i^{mis}, G', P'))\log P_{i,kj}. 
$$
Under the MAR assumption $P(X_i,Pa_i|Z_i,Z_{Pa_i},G',P') = P(X_i,Pa_i|G',P')$, we have $E(n_{i,kj}^{mis}) = E_{n_{i,j}^{mis}}(n_{i,j}^{mis}P_{i,kj}') = n_{i}^{mis}P_{i,j}'P_{i,kj}'$. Therefore 
\begin{equation}\label{eq:em_q}
Q^*(G,P|G',P') = \sum_{i=1}^N \sum_{j\in Pa_i, k\in X_i} (n_{i,kj} + (n-n_i)P_{i,j}'P_{i,kj}') \log P_{i,kj}. 
\end{equation}
We then observe that $P \mapsto Q^*(G, P|G, P)$ is maximized for 
$\hat P_{i,j} = n_{i,j}/n_i = \hat \theta_{i,j}$ and $\hat P_{i,kj} = n_{i,kj}/n_{i,j} = \hat\theta_{i,kj}$, and consequently   
$$
Q^*(\hat G,\hat P|\hat G,\hat P) = n \sum_{i=1}^N \sum_{j\in Pa_i}\hat P_{i,j}\sum_{k\in X_i}\hat P_{i,kj}\log\hat P_{i,kj}  = n l(\hat G|D_n). 
$$
We hence conclude that the EM algorithm based on $Q^*$ essentially maximizes the NAL function \eqref{eq:avrloglik}. 
Of course, $Q$ utilizes all of the available data, while $Q^*$ does not - when even one component of $(X_i,Pa_i)$ is missing, $Q^*$ treats the entire record as missing, while $Q$ tries to use the available information by calculating (often costly) conditional expectations. Nevertheless, the NAL-based inference is much more efficient than the naive approach that ignores all records for which at least one component of ${\bf X}$ is missing; even more so in cases when the dimensionality $N$ is much higher that the maximum size of $|Pa_i|$'s (the so-called in-degree). In such cases the difference between $Q$ and $Q^*$ is less pronounced (if $n_i>>n-n_i$ then $|Q-Q^*|<<|Q^*|$) and so is the difference between NAL maximization and EM algorithm. Moreover, the sub-optimality of NAL maximization is counterbalanced by its computational simplicity. 
The EM algorithm is usually intractable for data with number of nodes in the thousands while NAL optimization may still be a possibility. In conclusion, the NAL-based learning seems to be an effective and computationally more affordable alternative of EM for estimating high dimensional, low in-degree Bayesian networks.  

%%%%%%%%%%%%%%%%%%%%%%%%%%%%%%%%%%%%%%%%%%%%%%%%%%%%%%%%%%%%%%%%%%%%%%%%%%%%%%%%%%%%%%%%%%%%%%%%%%%%%%%%%%%%%%%%%%%%%%%%%%%%%
\section{MLE and model selection}\label{ch:model_selection}

Let $(G_0, P_0)$ be a DBN with nodes ${\bf X}$, parent sets $Pa_i^0$, and observed conditional probability table $\theta_0$. For an arbitrary DAG $G$ with nodes ${\bf X}$ and parents $Pa_i$, 
we consider probability distribution $P_{G|G_0}$ on ${\bf X}$ induced by $G_0$ which, for a state $x$ of ${\bf X}$, is given by 
\begin{equation}\label{eq:pgg0}
P_{G|G_0}(x) \equiv \prod_{i=1}^N P_0(X_i=x_i|Pa_i=x_{Pa_i}) 
\end{equation}
and compare it to 
$$
P_{0}(x) = \prod_{i=1}^N P_0(X_i=x_i|Pa_i^0=x_{Pa_i^0}).
$$
In general, $P_{G|G_0}$ is different from $P_0$ and $(G, P_{G|G_0})$ may not be well defined DBN, because $G$ is not necessarily a minimal DAG compatible with $P_{G|G_0}$ (see condition (2) from the definition of DBN). However, if $G$ is a minimal DAG such that $P_{G|G_0} = P_0$, then $G\cong G_0$. 

We also consider the following observation probabilities of $G$ induced by $G_0$ 
$$
\theta_{i}(G|G_0) \equiv P(Z_i=1, Z_{Pa_i}=1)
$$
$$
\theta_{i,j}(G|G_0) \equiv P(Pa_i=j|Z_i=1, Z_{Pa_i}=1)
$$
$$
\theta_{i,kj}(G|G_0) \equiv P(X_i=k|Pa_i=j, Z_i=1, Z_{Pa_i}=1)
$$
where the probabilities are with respect to the joint distribution of {\bf Z} and ${\bf X}|(G_0, P_0)$. 
Recall that according to \eqref{def:thetas} the entries of $\theta_0$ are
$$
\theta_{i}^0 \equiv P(Z_i=1, Z_{Pa_i^0}=1)
$$
$$
\theta_{i,j}^0 = P(Pa_i^0=j|Z_i=1, Z_{Pa_i^0}=1)
$$
$$
\theta_{i,kj}^0 = P(X_i=k|Pa_i^0=j, Z_i=1, Z_{Pa_i^0}=1). 
$$
Let $\theta(G|G_0)$ denote the corresponding conditional probability table with entries $\theta_i(G|G_0)$, $\theta_{i,j}(G|G_0)$ and $\theta_{i,jk}(G|G_0)$. Clearly, we can write $\theta(G_0|G_0) = \theta_0$. Moreover, in the important case when $Z$ is MCAR we have 
$$
\theta_{i,j}(G|G_0) \stackrel{mcar}{=} P_0(Pa_i=j)
$$
$$
\theta_{i,kj}(G|G_0) \stackrel{mcar}{=} P_0(X_i=k|Pa_i=j)
$$
and $\theta(G|G_0)$ is the conditional probability table corresponding to $P_{G|G_0}$. 

Next, we define the NAL of $G$ with respect to $G_0$ given by  
$$
l(G|G_0) \equiv \sum_{i=1}^N l(X_i|Pa_i, G_0) \textrm{, }
$$
\begin{equation}\label{eq:loglikG_G0}
l(X_i|Pa_i, G_0) \equiv \sum_{j\in Pa_i} \theta_{i,j} \sum_{k\in X_i} \theta_{i,kj} \log \theta_{i,kj}, 
\end{equation}
where $\theta_{i,j}=\theta_{i,j}(G|G_0)$ and $\theta_{i,kj}=\theta_{i,kj}(G|G_0)$. 
Essentially, $l(X_i|Pa_i, G_0)$ is the observed population negative entropy of $X_i$ conditional on $Pa_i$ and $l(G|G_0)$ is the population version of \eqref{eq:avrloglik}. 
For brevity, we shall write $l(G_0)$ instead of $l(G_0|G_0)$. 

%%%%%%%%%%%%%%%%%%%%%%%%%%%%%%%%%%%%%%%%%%%%%%%%%%%%%%%%%%%%%%%%%%%%%%%%%%%%%%%%%%%%%%%%%%%%%%%%%%%%%%%%%%%%%%%%%%%
\subsection{\small Identifiability}
\label{sec:ident}

Let $G_0$ belong to a collection $\mathcal{G}$ of DAGs with nodes ${\bf X}$. 
If $D_n$ is an independent sample from a DBN $(G_0, P_0)$, by the strong law of large numbers, for any fixed $G\in\mathcal{G}$, $\hat{\theta}_{i,kj}(G|D_n) \to \theta_{i,kj}(G|G_0)$, a.s., and hence, 
$l(G|D_n)\to l(G|G_0)$, a.s. as $n\to\infty$. 
A necessary condition for MLE consistency is the identifiability of $G_0$, which in its usual sense requires $l(G|G_0) < l(G_0)$ for all $G\in\mathcal{G}$ such that $G\ne G_0$. The latter is a strong requirement however, for thus defined the identifiability will never hold unless $G_0$ is a maximal DAG in $\mathcal{G}$ that contains $G_0$ - as we show later (Lemma \ref{lemma:lincrease}) $l(X_i|Pa_i, G_0)$ is a non-decreasing function of $Pa_i$. 
In the light of this observation we shall adopt a more appropriate definition of identifiability, one that assumes smaller likelihoods only for the DAGs not containing the true one. To simplify the notation, hereafter we shall refer to the DBN $(G_0,P_0)$ simply as $G_0$. 
\begin{deff}
\label{def:ident}
We say that $G_0$ is identifiable in $\mathcal{G}$, if for any $G\in\mathcal{G}$ we have $l(G|G_0) \leq l(G_0)$ when $G_0\subseteq G$ and $l(G|G_0) < l(G_0)$ when $G_0\nsubseteq G$. 
\end{deff}
Note that the identifiability of $G_0$ depends on the joint distribution of ${\bf X}$ and ${\bf Z}$. 
The utility of this definition is due to the following observation. If $G_0$ is identifiable in $\mathcal{G}$, then  
\begin{equation}\label{eq:identstar}
G^* \equiv \min\{\tilde G \in \mathcal{G}\textrm{ }|\textrm{ } l(\tilde G|G_0) = \max_{G\in\mathcal{G}}l(G|G_0)\} = G_0, 
\end{equation}
implicitly assuming the existence of unique such minimum $G^*$ (in general we may have multiple minimal $\tilde G$ maximizing the NAL). 
Moreover, it is easy to check that \eqref{eq:identstar} is a necessary and sufficient condition for identifiability. 
In `learning from data' settings, we can replace $l(G|G_0)$ in \eqref{eq:identstar} with $l(G|D_n)$ and find an estimator $\hat G^*$ of the minimal DAG $G^*$, exhaustively in $\mathcal{G}$ or by some more efficient algorithm. Then $\hat G^*$ would be an estimator of $G_0$ as well. In this way, the identifiability assures the principal possibility of recovering $G_0$.

It is intuitively clear that in order to recover the graph structure $G_0$ from incomplete samples, the missing data mechanism should not interfere with the associations between $X_i$'s determined by $G_0$. This condition is satisfied for any MCAR model. In more general MAR settings, the identifiability of $G_0$ depends on the interaction between ${\bf X}$ and ${\bf Z}$ and can not be judged without actually knowing $G_0$. We thus regard the MAR assumption as not significant generalization over MCAR due to the practical impossibility to check it prior to learning. 

The next result shows that in MCAR settings the population NAL does not increase when the true DBN is nested in a larger one, and moreover, that its maximum is achieved only for DAGs equivalent to the true one. 
\begin{pro}
\label{th:nal_increase}
If ${\bf Z}$ is MCAR, we have the following: 
\begin{enumerate}
\item[(i)]{if $G_0\subseteq G$ then $l(G|G_0)=l(G_0)$};
\item[(ii)]{$\max_{G} l(G|G_0) = l(G_0)$}, where the maximum is over all DAGs on ${\bf X}$;
\item[(iii)]{if $l(G|G_0) = l(G_0)$, then $P_{G|G_0} = P_0$}.
\end{enumerate}
\end{pro}

From these properties of the NAL of $G$ with respect to $G_0$ we can draw two immediate conclusions as stated in the next two corollaries. 
\begin{cor}
\label{cor:ident_inorder}
If ${\bf Z}$ is MCAR then $G_0$ is identifiable in any set of DAGs compatible with its order. 
\end{cor}
Therefore, provided a true node order is known (that is an order with which $G_0$ is compatible; there might be many such orders), $G_0$ can be recovered from the set of all DAGs compatible with that order. 

We can further extend Definition \ref{def:ident} to account for classes of equivalent DBNs. Recall that, ultimately, it is the independence relation set $\mathcal{I}(G_0)$, shared among all equivalent to $G_0$ DBNs, that is of main interest. 
In the view of condition \eqref{eq:identstar}, we say that $[G_0]$ is identifiable in $\mathcal{G}$ if 
\begin{equation}\label{eq:identstar2}
\min\{\tilde G \in \mathcal{G}\textrm{ }|\textrm{ } l(\tilde G|G_0) = \max_{G\in\mathcal{G}}l(G|G_0)\} \cong G_0, 
\end{equation}
in the sense that any minimal $\tilde G$ that maximizes the NAL $l(G|G_0)$ is equivalent to $G_0$ (we also assume that the set on the left is not empty). 
Proposition \ref{th:nal_increase}, cases $(ii)$ and $(iii)$, implies that the maximum NAL is $l(G_0)$ and any $G$ that attains this maximum satisfies $P_{G|G_0} = P_0$. If in addition $G$ is minimal, then $(G,P_{G|G_0})$ is a well defined DBN which is equivalent to $(G_0,P_0)$ and hence \eqref{eq:identstar2} is satisfied. We have thus obtained the following. 
\begin{cor}
\label{cor:ident_eqclass}
If ${\bf Z}$ is MCAR, then $[G_0]$ is identifiable in any $\mathcal{G}$ that contains at least one element of $[G_0]$. In particular, $[G_0]$ is (globally) identifiable in the set of all DAGs on ${\bf X}$. 
\end{cor}
As defined, the identifiability of the equivalence class $[G_0]$ depends implicitly on the choice of log-likelihood proxy function. Note that $[G_0]$ is not guaranteed to be identifiable, even in MCAR settings, if in \eqref{eq:identstar2} we replace the NAL $l$ with the standard log-likelihood $\tilde l$ from \eqref{eq:totalloglik}. 

%%%%%%%%%%%%%%%%%%%%%%%%%%%%%%%%%%%%%%%%%%%%%%%%%%%%%%%%%%%%%%%%%%%%%%%%%%%%%%%%%%%%%%%%%%%%%%%%%%%%%%%%%%%%%%%%%%%
\subsection{\small NAL-based scoring functions}

As we have observed earlier, the MLE criteria selects the most complex BN in $\mathcal{G}$ containing $G_0$ and unless some complexity penalization is imposed, the MLE is prone to overfitting. 
Methodologically, there are two approaches addressing the model selection problem. The first one is provided by the Bayesian paradigm, where the parameter $(G,\theta)$ is assumed coming from some prior distribution and one looks for the maximum posterior estimator. The second, frequentist, approach is to optimize a scoring function based on the log-likelihood and additional complexity penalization term - a penalized log-likelihood. We consider a general scoring function of the form 
\begin{equation}\label{eq:score}
S(G|D_n) = l(G|D_n) - \lambda_n h(G), 
\end{equation}
where $\lambda_n$ are positive numbers indexed by the sample size $n$ and $h(G)$ is a positive function accounting for the complexity of the $G$. When needed, we shall write $S_h$ to specify what $h$ is meant. 
The role of the sequence $\lambda_n$ is to apply a proper amount of penalty that guarantees estimation consistency. 

One can employ different measures for network complexity. Any complexity function $h$ is assumed to be increasing in the following sense: for any two DAGs $G_1$ and $G_2$ such that $G_1\subset G_2$, $G_1\ne G_2$, we have $h(G_1) < h(G_2)$. In regard to DBNs, a typical choice is the total number of parameters $df(G)$ needed to specify the multinomial conditional distributions of $G$, that is, the number of independent parameters in $\theta$. 

We return to \eqref{eq:score} with some typical examples. Since the NAL $l(G|D_n)$, being sum of node sample averages, is normalized by the sample size, the standard model selection criteria AIC and BIC, formulated in terms of the scoring function \eqref{eq:score} are given by 
$\lambda_n = 1/n$ and $\lambda_n = 0.5\log(n)/n$, respectively. 
The so called minimum description length (MDL) score, representing the information content of a model, is given by $\log(n)df(G)/n$ and is equivalent to BIC. 

Similarly to NAL, often, the chosen overall DBN complexity can also be represented as a sum of node-wise complexities. For example, $df(G) = \sum_i df(X_i|Pa_i)$, $df(X_i|Pa_i) \equiv (q(X_1)-1)q(Pa_i)$. In such cases it might be more appropriate to replace $\lambda_n$ with node-specific penalization $\lambda_{n_i}$'s 
\begin{equation}\label{eq:score_n}
S(G|D_n) = \sum_{i=1}^N \{ l(X_i|Pa_i,D_n) - \lambda_{n_i} h(X_i|Pa_i) \}. 
\end{equation} 
We shall refer to these as decomposable scores. 
Typically, one uses one and the same function of $n$ to express $\lambda_{n_i}$'s, such as $\lambda_n = \lambda_0n^{-\alpha}$, $\alpha\in (0,0.5)$. 
The decomposable BIC criteria then is 
$$
S_{BIC}(X_i|Pa_i, D_{n,i}) = l(X_i|Pa_i,D_n) - 0.5\frac{\log(n_i)}{n_i} df(X_i|Pa_i)
$$
\begin{equation}\label{eq:decompBIC}
S_{BIC}(G|D_n) = \sum_{i=1}^N S_{BIC}(X_i|Pa_i, D_n) = \sum_{i=1}^N \frac{1}{n_i} BIC(X_i|Pa_i, D_{n,i})
\end{equation}
where $D_{n,i}$ is the sub-sample of $D_n$ of size $n_i$ for which $(X_i,Pa_i)$ is observed and $BIC(X_i|Pa_i,D_{n,i})$ is the original BIC criteria, \eqref{eq:BIC}, applied to the regression model $X_i|Pa_i$. 

As we have stated in the introduction, we consider an MLE based model selection by maximizing $S$ as a function of $G$ given a sample $D_n$, 
\begin{equation}\label{eq:G_hat}
\hat{G} = arg\max_{G\in\mathcal{G}} S(G|D_n).
\end{equation}
Note that we do not maximize $S$ for $G$ and $\theta$ simultaneously. 
We estimate $\theta$ for each $G$ using the plug-in estimator $\hat{\theta}(G|D_n)$ and then the DAG with maximal score is chosen as graph structure estimator. 
In what follows we show that, by solving \eqref{eq:G_hat} for proper $\lambda_n$, we can obtain consistent estimation of the true model with no further conditions on $h$. 

Let $\hat G$ be the estimator \eqref{eq:G_hat} for a sample $D_n$ coming from a DBN $G_0$. Then the following claim is immediate. 
\begin{pro}[{\bf Consistency Criteria}]
\label{th:mle_cons}
Provided for any $G_1\in \mathcal{G}$ and $G_2\in \mathcal{G}$ the following two conditions are satisfied  
\begin{enumerate}
\item[(C1)]{if $G_0\subseteq G_1$ but $G_0 \nsubseteq G_2$, then $P(S(G_1|D_n) > S(G_2|D_n)) \to 1$, as $n\to\infty$, 
}
\item[(C2)]{if $G_0\subseteq G_1$, $G_0 \subset G_2$ and $h(G_1) < h(G_2)$, then $P(S(G_1|D_n) > S(G_2|D_n)) \to 1$, as $n\to\infty$, 
}
\end{enumerate}
$\hat G$ is a consistent estimator of $G_0$, that is, $P(\hat G \ne G_0) \to 0$, as $n\to\infty$.
\end{pro}
The conditions (C1) and (C2) are relaxed versions of those used in \cite{chickering}. 
In fact, the consistent scoring criterion in \cite{chickering} is a special case of the more abstract formulation of model selection consistency in \cite{haughton}. 
We end this section with the following important observation. 
\begin{cor}
\label{cor:equ_cons} 
If conditions (C1) and (C2) are satisfied for any DAG equivalent to $G_0$, then 
$[\hat G]$ is a consistent estimator of $[G_0]$, that is, $P(\mathcal{I}(\hat G) \ne \mathcal{I}(G_0)) \to 0$, as $n\to\infty$.
\end{cor}

%%%%%%%%%%%%%%%%%%%%%%%%%%%%%%%%%%%%%%%%%%%%%%%%%%%%%%%%%%%%%%%%%%%%%%%%%%%%%%%%%%%%%%%%%%%%%%%%%%%%%%%%%%%%%%%%%%%
\section{Estimation consistency}\label{ch:consistent_scores}

Let $(G_0,P_0)$ be a DBN with conditional table $\theta_0$ in a set of DAGs $\mathcal{G}$ and $D_n$ be an independent sample drawn from it. In this section we investigate the consistency of the estimators $\hat{G}$ and $[\hat{G}]$ with respect to a scoring function $S$, where $\hat{G}$ is given by \eqref{eq:G_hat}. 

As we have observed earlier, if the data has missing values, it is not anymore true that $l(G|D_n)=\tilde{l}(G|D_n)$, the usual sample average log-likelihood \eqref{eq:totalloglik}. Therefore, $(\hat{G}, \hat{\theta})$ is no longer an MLE for $(G_0,\theta_0)$ and the standard consistency results from the asymptotic theory are not  directly applicable. A proper account for the incompleteness of the data is thus needed. 

For a sample of fixed size $n$, the random variables $n_i$ and the random vectors $\{n_{i,j}\}_{j\in Pa_i}$ and $\{n_{i,kj}\}_{k\in X_i}$ satisfy 
$$
\{n_{i,j}\}_j|n_i \sim Multinom(\{\theta_{i,j}(G|G_0)\}_j, n_i)
$$
\begin{equation}\label{eq:assumption_cond_distr}
\{n_{i,kj}\}_k|n_{i,j} \sim Multinom(\{\theta_{i,kj}(G|G_0)\}_k, n_{i,j})
\end{equation}
and the statistics $\hat{\theta}_{i,j}$ and $\hat{\theta}_{i,kj}$ are unbiased estimators of $\theta_{i,j}(G|G_0)$ and $\theta_{i,kj}(G|G_0)$, respectively. 
Moreover, if $G_0$ is identifiable in $\mathcal{G}$, then for each $i$, the probability of the event `$(X_i,Pa_i^0)$ is observed' must be strictly  positive, i.e. $\theta_i^0 > 0$. Since $\mathcal{G}$ is always finite, the following is well defined  
\begin{equation}\label{eq:assumption_beta}
\beta(\mathcal{G}) \equiv \min_{G\in\mathcal{G}} \min_{i=1}^{N} \{\theta_i(G|G_0)| \theta_i(G|G_0) > 0\} 
\end{equation}
and $\beta(\mathcal{G}) > 0$. 
The complete data case can be thus represented as $\beta(\mathcal{G}) = 1$. 
Note that $\beta$ depends implicitly on the distribution of ${\bf Z}$.

The next result establishes the rate of convergence of the empirical NAL to the population one without imposing any restrictions on the distribution of ${\bf Z}$ or on $G_0$ ($G_0$ need not be identifiable). 
\begin{lem}
\label{th:l_unif_conv}
Let $D_n$ be sample from a DBN $(G_0, P_0)$. Then for any DAG $G$ 
\begin{equation}\label{eq:PANM}
l(G|D_n)-l(G|G_0) = O_p(n^{-1/2}),
\end{equation}
which implies $l(G|D_n) \to_p l(G|G_0)$.
\end{lem}

Providing conditions for scoring function consistency is our next goal. 
Let us assume that $G_0$ is identifiable in $\mathcal{G}$. 
In the light of Lemma \ref{th:l_unif_conv}, if $G$ does not contain $G_0$, then there is a positive constant $\delta$ such that $l(G_0|D_n) - l(G|D_n) > \delta$ with probability going to 1, as $n\to\infty$. It is evident therefore that if the sequence $\lambda_n$ diminishes with $n$, $\lambda_n\to 0$, then, asymptotically, the scoring function $S_h$ will select an estimator that contains the true model $G_0$ regardless of the chosen complexity function $h$. In addition however, we want that estimator to get close (in sense of the complexity measured by $h$) to $G_0$ with the increase of the sample size. 
Since for any $G$ such that $G_0\subset G$ we have $l(G|D_n) - l(G_0|D_n) \to_p 0$, the latter can be assured if we require $\lambda_n$ to diminish at a slower rate than that of $l(G|D_n) - l(G_0|D_n)$. We show that this rate is $n^{-1}$ for complete samples and $n^{-1/2}$ in case of missing data. 

We moreover show that the consistency sufficient conditions, $\lambda_n=o(1)$ and $n^{-1/2}\lambda_n^{-1}=o(1)$, become essentially necessary. More precisely, the necessity is guaranteed if the following condition is satisfied. As usual $Pa_i$ and $Pa_i^0$ denote the parent sets of $G$ and $G_0$, respectively. 
\begin{con}
\label{eq:assumption_incons}
There are $G\in\mathcal{G}$ with $G_0 \subset G$ and $i\in\{1,...,N\}$ such that $Pa_j=Pa_j^0$ for all $j\ne i$, $Pa_i\backslash Pa_i^0 \ne \emptyset$ and $P(Z_{Pa_i\backslash Pa_i^0} = 1 | Z_i=1,Z_{Pa_i^0}=1) \in (0,1)$.
\end{con}
In words, the condition refers to the possibility of extending the parent set of a node of $G_0$ by one or more new nodes that are, conditionally, neither always observed nor never observed (thus $G_0$ must not be a maximal DAG in $\mathcal{G}$). 

%%%%%%%%%%%%%%%%%%%%%%%%%%%%%%%%%%%%%%%%%%%%%%%%%%%%%%%%%%%%%%%%%%%%%%%%%%%%%%%%%%%%%%%%%%%%%%%%%%%%%%%%%%%%%%%%%

Next, we summarize the above observations in the following theorem. 
\begin{thm}
\label{th:consistency}
Let $G_0$ be identifiable in $\mathcal{G}$ and $S$ be a scoring function \eqref{eq:score} with penalization parameter $\lambda_n$ such that $\lambda_n\to 0$.  The following are satisfied. 
\begin{enumerate}
\item[(i)]{If $\beta(\mathcal{G}) \in (0,1)$ and $\sqrt{n} \lambda_n\to \infty$, then $\hat G$ is consistent estimator of $G_0$. 
}
\item[(ii)]{If $\beta(\mathcal{G}) = 1$ and $n \lambda_n\to \infty$, then $\hat G$ is consistent estimator of $G_0$. 
}
\item[(iii)]{If ${\bf Z}$ is MCAR, Condition \ref{eq:assumption_incons} holds and $\underline{\lim}\sqrt{n} \lambda_n < \infty$, then $\hat G$ is inconsistent estimator of $G_0$. 
}
\end{enumerate}
\end{thm}
The complete data case of the theorem, $(ii)$, also follows from a more general result by \cite{haughton} (Proposition 1.2 and Remark 1.2). There, the consistency result is derived using the properties of MLE for exponential families and central limit theorem. The essential contribution of the above theorem is in the missing data cases $(i)$ and $(iii)$. We emphasize that case $(i)$ holds for a general $\mathcal{G}$ and missing data distribution as long as $G_0$ is identifiable in $\mathcal{G}$. In $(iii)$ however, we require for ${\bf Z}$ to be MCAR in order to guarantee that the condition $\sqrt{n} \lambda_n\to \infty$ is necessary for consistent estimation. 
Below we make some further remarks. 

The claims of the theorem are established by verifying conditions (C1) and (C2) from Proposition \ref{th:mle_cons} for $G_0$ and hence, for any DAG equivalent to $G_0$. Therefore, it follows from Corollary \ref{cor:equ_cons} that the theorem remains true if we replace $G_0$ by $[G_0]$ and $\hat G$ by $[\hat G]$. The theorem thus provides conditions for consistent estimation of the equivalence class of $G_0$. 

As evident from the proof of the theorem, the requirement $\lambda_n\to 0$ is needed for guaranteeing the first, (C1), consistency condition in Proposition \ref{th:mle_cons}, while $\sqrt{n}\lambda_n\to \infty$ ($n\lambda_n\to \infty$) is required for the second one (C2). The AIC selection criterion, $\lambda_n = 1/n$, is not consistent for it satisfies (C1) but fails to satisfy (C2), regardless of $\beta$. It will thus recover the true structure but will tend to select networks with higher complexities than the true one. Therefore AIC is prone to overfitting and so is any scoring function with $n\lambda_n = O(1)$. At the other end of the consistency spectrum of $\alpha$, $\lim_n\sup \lambda_n > 0$, the estimated networks will tend to have complexities below the true one. Due to the missingness, there is an implication regarding the BIC(MDL) criterion, $\lambda_n = 0.5 \log(n)/n$. Because $n\log(n)/n\rightarrow \infty$ but $\sqrt{n}\log(n)/n\rightarrow 0$, BIC is guaranteed to be consistent only in the complete data case and it will be, in general, inconsistent in MCAR settings (see the corollaries that follow). The numerical results presented in Section \ref{ch:experiments} confirm this conclusion. 

Theorem \ref{th:consistency} requires the observation probability $\beta(\mathcal{G})$ to be fixed. If we allow it to depend on $n$, case $(ii)$ of the theorem arises from $(i)$ if we have $\lim_n \beta_n(\mathcal{G}) = 1$. Then $n\lambda_n\to\infty$ is a sufficient consistency condition. There is no contradiction with case (iii), since then it must be that $P(Z_{X_i}=1, Z_{Pa_i}=1) = 1$ and $P(Z_{X_i}=1, Z_{Pa_i^0}=1) = 1$, and hence Condition \ref{eq:assumption_incons} fails. As evident from the proof of the theorem, when $\overline\lim\beta_n < 1$, $(i)$ and $(iii)$ still hold. We leave undecided the last alternative $0 < \underline\lim\beta_n < \overline\lim\beta_n = 1$. 

Next, we argue that Condition \ref{eq:assumption_incons} arises naturally in MCAR settings. 
In the probability space of all MCAR distributions for ${\bf Z}$ defined by the Borel sets in
$\{u\in[0,1]^{2^N-1}, \sum_{k=1}^{2^N-1} u_k \le 1 \}$ 
(a distribution $u$ is defined by assigning each of the $2^N$ states of ${\bf Z}$ a probability value in [0,1] such that their sum is 1), the subspace of distributions for which $P(Z_{Pa_i\backslash Pa_i^0} = 1 | Z_i=1,Z_{Pa_i^0}=1)=0\textrm{ or }1$ has Borel measure zero. We thus have the following consequences of Theorem \ref{th:consistency} which extend Corollary \ref{cor:ident_inorder} and \ref{cor:ident_eqclass}, and essentially summarize the practical contribution of this investigation. 
\begin{cor}
\label{th:suffcond}
Let $(G_0, P_0)$ be a non-maximal DBN and $\mathcal{G}$ consist of all DAGs compatible with a node order of $G_0$. Then, for almost all MCAR distributions, $\hat G$ is consistent estimator of $G_0$ if and only if $\lambda_n\to 0$ and $\sqrt{n} \lambda_n\to \infty$. 
\end{cor}
In the last statement we assume that $\mathcal{G}$ comprises all DAGs on ${\bf X}$ and use the global identifiability of $[G_0]$. 
\begin{cor}
\label{th:suffcondeqclass}
Provided that $\mathcal{I}(G_0)$ is non-empty, for almost all MCAR distributions, $[\hat G]$ is consistent estimator of $[G_0]$ if and only if $\lambda_n\to 0$ and $\sqrt{n} \lambda_n\to \infty$. 
\end{cor}

Note that, the non-emptiness of $\mathcal{I}(G_0)$ is required in order for any DAG equivalent to $G_0$ to be non-maximal and hence, for Condition \ref{eq:assumption_incons} to hold. 

\section{Numerical experiments}\label{ch:experiments}

With the number of possible DAGs being super-exponential to the number of nodes, the task of reconstructing a DBN from data is in general NP-hard. The MLE based problem \eqref{eq:G_hat} essentially requires exhausting all DAGs in $\mathcal{G}$. For the purpose of numerical illustration in this section we make two simplifying the inference assumptions - that the causal order of the nodes of the original DBN $G_0$ is known, as well as the maximum size of the parent sets of $G_0$, its in-degree. We thus assume that the search set $\mathcal{G}$ comprises all DAGs compatible with a true node order. By Corollary \ref{cor:ident_inorder}, when the missing data model is MCAR, $G_0$ is identifiable in $\mathcal{G}$. In our numerical experiments we use exclusively the complexity function $df$, which recall is given by $df(G,\theta) = dim(\Theta_G)$, and the decomposable scoring function \eqref{eq:score_n}. Then \eqref{eq:G_hat} can be solved by an efficient exhaustive search via dynamic programming, an approach that is implemented in the {\it catnet} package for {\it R}, \cite{catnet}. We are aware that more general learning algorithms are available in the literature that can also accommodate available case analysis based on NAL. For example, one can implement a search based on local optimizations as described in \cite{chickering} by replacing the usual log-likelihood with NAL. However, our goal here is not to compare different learning strategies but to empirically verify the conclusions of Theorem \ref{th:consistency}, which hold for all NAL-based estimators \eqref{eq:G_hat}. 

The standard AIC and BIC model selection criteria are compared to scoring functions with $\lambda_n = (1/N)n^{-\alpha}$ for different choices of $\alpha\in(0,1)$. The factor $1/N$, to some extent arbitrary, makes the penalization relatively small for not large $n$ (note that NAL is of rate $O(N)$). For this choice of $\lambda_n$ and small $n$, the estimator $\hat{G}$ therefore may over-fit the data but, provided the scoring criteria is consistent, $df(\hat{G})$ should approach the true complexity as $n$ increases.  

\begin{table}
\caption{Consistency results for a simulated 2-node network. Two possible models $G_0$ and $G_1$ are considered as described in the main text. Shown are the percents of wrong selections (choosing the alternative $G_1$ instead of the true model $G_0$). }
\begin{center}
\begin{tabular}{|l|l|l|l|l|l|l|l|l|l|}
\hline
 $\alpha$ & 0.2 & 0.3 & 0.4 & 0.5 & 0.6 & 0.7 & 0.8 & BIC & AIC \\ \hline
 
 \multicolumn{10}{|c|}{$n=10^2$} \\ \hline 
$\beta=1$ & 0.0 & 0.0 & 0.0 & 0.3 & 0.9 & 3.5 & 10.6 & 2.8 & 16.0 \\   
 $\beta=0.99$ & 0.0 & 0.0 & 0.0 & 0.5 & 1.7 & 6.6 & 17.0 & 4.5 & 22.9 \\   
 $\beta=0.95$ & 0.0 & 0.0 & 0.2 & 0.9 & 3.8 & 12.8 & 24.0 & 8.7 & 31.2 \\   
 $\beta=0.90$ & 0.0 & 0.0 & 0.0 & 0.7 & 6.9 & 16.6 & 31.5 & 12.5 & 37.0 \\   
 $\beta=0.75$ & 0.0 & 0.0 & 1.1 & 7.0 & 18.5 & 29.9 & 40.2 & 27.3 & 44.4 \\   \hline   
 \multicolumn{10}{|c|}{$n=10^3$} \\ \hline 
 $\beta=1$ & 0.0 & 0.0 & 0.0 & 0.0 & 0.0 & 0.2 & 3.6 & 0.7 & 13.9 \\   
 $\beta=0.99$ & 0.0 & 0.0 & 0.0 & 0.0 & 0.0 & 1.6 & 13.3 & 2.9 & 33.5 \\   
 $\beta=0.95$ & 0.0 & 0.0 & 0.0 & 0.0 & 0.4 & 12.1 & 28.8 & 17.1 & 42.0 \\   
 $\beta=0.90$ & 0.0 & 0.0 & 0.0 & 0.1 & 3.6 & 19.1 & 34.7 & 23.0 & 43.9 \\   
 $\beta=0.75$ & 0.0 & 0.0 & 0.0 & 1.9 & 15.8 & 33.2 & 42.4 & 36.2 & 47.2 \\   
 \hline   \multicolumn{10}{|c|}{$n=10^4$} \\ \hline 
 $\beta=1$ & 0.0 & 0.0 & 0.0 & 0.0 & 0.0 & 0.0 & 0.8 & 0.2 & 15.0 \\   
 $\beta=0.99$ & 0.0 & 0.0 & 0.0 & 0.0 & 0.0 & 2.7 & 24.7 & 13.9 & 44.1 \\   
 $\beta=0.95$ & 0.0 & 0.0 & 0.0 & 0.0 & 1.5 & 21.3 & 37.7 & 31.6 & 47.8 \\   
 $\beta=0.90$ & 0.0 & 0.0 & 0.0 & 0.0 & 7.0 & 28.9 & 41.5 & 36.5 & 47.5 \\   
 $\beta=0.75$ & 0.0 & 0.0 & 0.0 & 1.8 & 22.3 & 41.2 & 50.5 & 47.3 & 53.8 \\   
 \hline   \multicolumn{10}{|c|}{$n=10^5$} \\ \hline 
 $\beta=1$ & 0.0 & 0.0 & 0.0 & 0.0 & 0.0 & 0.0 & 0.0 & 0.0 & 17.7 \\   
 $\beta=0.99$ & 0.0 & 0.0 & 0.0 & 0.0 & 0.0 & 11.8 & 37.8 & 35.8 & 50.4 \\   
 $\beta=0.95$ & 0.0 & 0.0 & 0.0 & 0.0 & 3.5 & 31.3 & 44.6 & 43.3 & 49.8 \\   
 $\beta=0.90$ & 0.0 & 0.0 & 0.0 & 0.0 & 13.1 & 36.1 & 47.2 & 46.0 & 50.5 \\   
 $\beta=0.75$ & 0.0 & 0.0 & 0.0 & 1.0 & 21.4 & 38.5 & 45.5 & 45.3 & 48.2 \\   
 \hline  
\end{tabular}
\end{center}
\label{table:abtest}
\end{table}

\subsection{Simulated 2-node network}

Here we consider a simplest possible example to verify the consistency of the NAL estimator \eqref{eq:G_hat}. 
We generate samples from a model $G_0$ with 2 independent binary variables $X_1$ and $X_2$ (that is $Pa_1=Pa_2=\emptyset$) with marginal probabilities $\theta_1=(0.4,0.6)$ and $\theta_2=(0.3,0.7)$, respectively. We assume that $X_1\prec X_2$ and then the only alternative to $G_0$ BN model is $G_1$ with $Pa_1=\emptyset$ and $Pa_2=\{X_1\}$. We also assume that $X_2$ is always observed ($P(Z_2=1)=1$) but $Z_1$ is MCAR with different missing probabilities $P(Z_1=0) \in \{0, 0.01, 0.05, 0.10, 0.25\}$. The sample observation probability is then $\beta = 1-P(Z_1=0)$.
For each $\beta$ and sample size $n\in\{10^2,10^3,10^4,10^5\}$, we generate 1000 samples $D_n$ and count how many times $S(G_1|D_n) > S(G_0|D_n)$, that is, $\hat G = G_1$ and $G_1$ is erroneously selected instead of $G_0$. Table \ref{table:abtest} summarizes results for different choices of the penalization parameter $\alpha$ as well as BIC and AIC. As expected, all considered scoring functions except AIC are consistent in the no missing data case ($\beta=1$). In presence of missing values however, for scoring functions with $\alpha>0.5$ the percent of false model selections is significant; moreover, it increases when the proportion of missing values increases, suggesting inconsistency. In particular, the inconsistency of BIC is very pronounced for all $\beta<1$.  Even when the proportion of missing values is only 1 percent, $\beta=0.99$, the percent of wrong selections start from 4.5 for $n=10^2$ and climbs to 35.8 for $n=10^5$. The presented results are in strong support of the predictions of Theorem \ref{th:consistency}. 

\subsection{Consistent estimation of the ALARM network}
\label{sec:alarm}

Here we consider a well known in the literature benchmark network. 
ALARM, a medical diagnostic alarm message system for patient monitoring developed by \cite{beinlich}, is a typical example of belief propagation network as those employed in many expert systems. 
The DAG of ALARM has 37 nodes, 45 directed edges, varying number of categories (2,3 and 4) and complexity of 473.
We perform network reconstruction using both complete and MCAR missing data simulated from the network, in order to confirm the effect of missingness on the model selection as predicted by Theorem \ref{th:consistency}. The graph structures of the estimated networks are compared to the original one by the so-called $F$-score, the harmonic mean of precision (${TP}$/$(TP+FP)$) and recall (${TP}$/$(TP+FN)$), where $P$ and $N$ refer to the presence and absence of directed edges. $F$-score of 1 represents perfect reconstruction. 

Missing data samples are simulated by deleting $1$, $2$ and $4$ values from each sample record, completely at random (so, there is not even 1 fully complete record in the samples). Since the maximum parent size is 3, in the first case, the probability to have no missing 3-node subset $(X_i,Pa_i)$ is ${36 \choose 3}/{37 \choose 3}$ and hence, the effective observation probability $\beta$ from \eqref{eq:assumption_beta} is about $0.92$. 
When 2 values per record are deleted, $\beta$ drops to ${35 \choose 3}/{37 \choose 3} \approx 0.84$; when 4 values are deleted, $\beta$ is about $0.70$. 
The target set of models $\mathcal{G}$ includes all DAGs with 37 nodes, maximum of 3 parents per node, compatible with the true node order. Under these constraints, the number of DAGs in $\mathcal{G}$ is 1133. For each possible complexity $t$, the optimal network $\hat{G}(t)$ 
is found and a final selection is made according to their scores. 

Table \ref{table:alarm_results} shows comparison results for 9 scoring criteria 
($\alpha$=0.25,0.3,0.35,0.4,0.45,0.5, 0.75, BIC and AIC) and 7 samples sizes, from 5e2 to 2.5e5. In the complete data case, the scoring functions with $\alpha\in[0.4,0.5]$ and BIC reconstruct the true network for all samples with $n\ge 2.5e4$. As predicted, in the missing data cases the score function for $\alpha=0.5$ and BIC become inconsistent due to overfitting. 
This effect is more clearly demonstrated in Figures \ref{fig:alarm_profile_25e4} and \ref{fig:alarm_profile_n} that show the complexity profile functions for different experimental cases. According to Theorem \ref{th:consistency}, the complexity profiles in the $(0,0.5)$ range should converge to the horizontal line of true complexity. In Figure \ref{fig:alarm_profile_25e4} the sample size is kept fixed and we see that with the increase of the proportion of missing values ($\beta$ decreasing), the profiles depart from the line of true complexity. On the other hand, Figure \ref{fig:alarm_profile_n} shows profiles of samples with fixed proportion of missing values but of increasing size. We observe that, although slow, the profiles get closer to the line of true complexity as $n$ increases. 
It is also evident that the BIC selected complexity drifts up and away from the true one with the increase of the sample size, an indication for its inconsistency. 

\begin{table}
\caption{Model selection results for the ALARM network using complete samples ($\beta=1$) and samples following MCAR models with $\beta=0.84$ and $\beta=0.70$. Shown are the F-scores between the true network and the estimated ones. }
\begin{center}
\begin{tabular}{|l|l|l|l|l|l|l|l|}
\hline
 n & 5e2 & 2.5e3 & 5e3 & 2.5e4 & 5e4 & 1e5 & 2.5e5 \\ \hline
 \multicolumn{8}{|c|}{no missing values, $\beta=1$} \\ \hline 
 $\alpha$ = 0.25 & 0.85 & 0.95 &      0.97  &      0.97  &      0.97  &      0.97  &      0.98 \\   
 $\alpha$ = 0.3  & 0.88 & 0.97 &      0.97  &      0.97  &      0.98  &      0.98  &      0.99 \\   
 $\alpha$ = 0.35 & 0.86 & 0.97 &      0.97  &      0.98  &      0.99  & {\bf 1.00} & {\bf 1.00} \\  
 $\alpha$ = 0.4  & 0.81 & 0.97 &      0.98  & {\bf 1.00} & {\bf 1.00} & {\bf 1.00} & {\bf 1.00} \\   
 $\alpha$ = 0.45 & 0.75 & 0.96 & {\bf 1.00} & {\bf 1.00} & {\bf 1.00} & {\bf 1.00} & {\bf 1.00} \\   
 $\alpha$ = 0.5  & 0.73 & 0.92 &      0.99  & {\bf 1.00} & {\bf 1.00} & {\bf 1.00} & {\bf 1.00} \\   
 $\alpha$ = 0.75 & 0.57 & 0.65 &      0.62  &      0.62  &      0.65  &      0.63  &      0.65 \\   
 BIC             & 0.85 & 0.97 &      0.98  & {\bf 1.00} & {\bf 1.00} & {\bf 1.00} & {\bf 1.00} \\   
 AIC             & 0.80 & 1.84 &      0.82  &      0.81  &      0.79  &      0.80  &      0.80 \\ \hline   
 \multicolumn{8}{|c|}{MCAR, $\beta=0.84$} \\ \hline 
 $\alpha$ = 0.25 & 0.79 & 0.86 & 0.91 & 0.97 & 0.97 & 0.97 & 0.97 \\   
 $\alpha$ = 0.3  & 0.79 & 0.88 & 0.88 & 0.93 & 0.93 & 0.97 & 0.99 \\   
 $\alpha$ = 0.35 & 0.79 & 0.82 & 0.85 & 0.90 & 0.92 & 0.95 & {\bf 1.00} \\   
 $\alpha$ = 0.4  & 0.74 & 0.78 & 0.81 & 0.85 & 0.86 & 0.91 & 0.92 \\   
 $\alpha$ = 0.45 & 0.69 & 0.72 & 0.77 & 0.80 & 0.79 & 0.83 & 0.83 \\   
 $\alpha$ = 0.5  & 0.65 & 0.70 & 0.70 & 0.72 & 0.71 & 0.78 & 0.74 \\   
 $\alpha$ = 0.75 & 0.56 & 0.60 & 0.62 & 0.61 & 0.61 & 0.63 & 0.61 \\   
 BIC             & 0.77 & 0.82 & 0.82 & 0.77 & 0.71 & 0.71 & 0.66 \\   
 AIC             & 0.74 & 0.68 & 0.67 & 0.62 & 0.61 & 0.63 & 0.61 \\ \hline 
 \multicolumn{8}{|c|}{MCAR, $\beta=0.70$} \\ \hline 
 $\alpha$ = 0.25 & 0.76 & 0.83 & 0.88 & 0.93 & 0.97 & 0.97 & 0.97 \\   
 $\alpha$ = 0.3  & 0.75 & 0.80 & 0.86 & 0.89 & 0.91 & 0.93 & {\bf 1.00} \\   
 $\alpha$ = 0.35 & 0.69 & 0.73 & 0.84 & 0.87 & 0.89 & 0.93 & 0.95 \\   
 $\alpha$ = 0.4  & 0.68 & 0.73 & 0.78 & 0.82 & 0.82 & 0.87 & 0.89 \\   
 $\alpha$ = 0.45 & 0.66 & 0.65 & 0.74 & 0.76 & 0.79 & 0.78 & 0.79 \\   
 $\alpha$ = 0.5  & 0.58 & 0.61 & 0.69 & 0.70 & 0.72 & 0.70 & 0.73 \\   
 $\alpha$ = 0.75 & 0.56 & 0.57 & 0.61 & 0.61 & 0.62 & 0.61 & 0.63 \\   
 BIC             & 0.74 & 0.75 & 0.79 & 0.74 & 0.74 & 0.69 & 0.67 \\   
 AIC             & 0.70 & 0.61 & 0.66 & 0.62 & 0.63 & 0.61 & 0.63 \\   
 \hline
\end{tabular}
\end{center}
\label{table:alarm_results}
\end{table}

\begin{figure}
\caption{Estimating the ALARM network from complete and MCAR samples of size $2.5e5$. For $\beta=1,0.92,0.84,0.70$, the so called complexity profile - the complexity of the estimated network on y-axis as a function of the penalization parameter $\alpha$ on x-axis - is shown in the range $[0.25,0.5]$. In presence of missing values ($\beta<1$), the BIC selection (dash, horizontal) tends to move up and away from the true complexity of 473 (solid, horizontal), demonstrating the inconsistency of BIC. }
\includegraphics[scale=0.35,angle=-90]{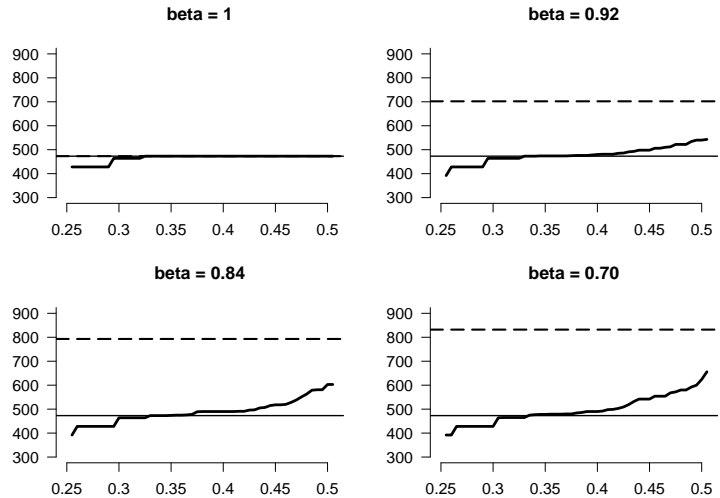}
\label{fig:alarm_profile_25e4}
\end{figure}

\begin{figure}
\caption{Estimating ALARM from MCAR samples with fixed observation probability $\beta=0.84$. Shown are the complexity profiles of 6 samples of increasing size $n$. For $n>1e6$, the complexity of the BIC estimates are off charts ($>900$). }
\includegraphics[scale=0.4,angle=-90]{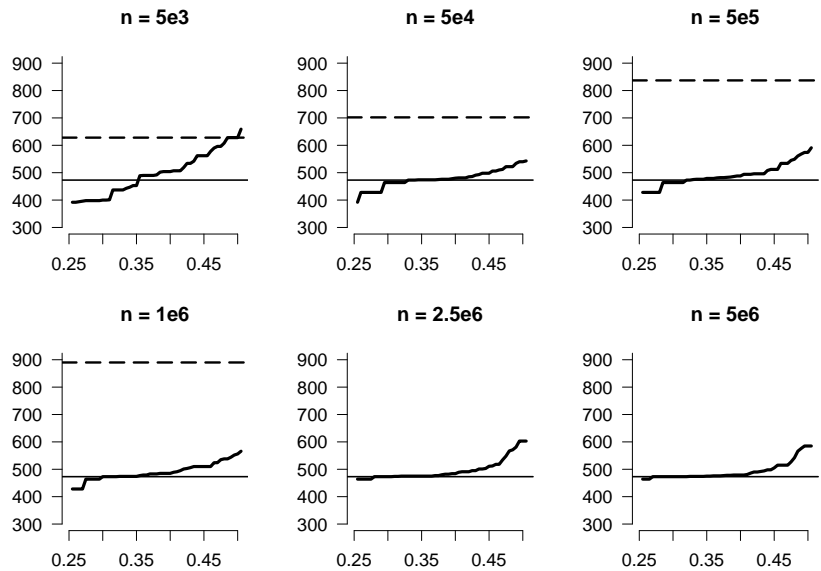}
\label{fig:alarm_profile_n}
\end{figure}

%%%%%%%%%%%%%%%%%%%%%%%%%%%%%%%%%%%%%%%%%%%%%%%%%%%%%%%%%%%%%%%%%%%%%%%%%%%%%%%%%%%%%%%%%%%%%%%%%%%%%%%%%%%%%%%%%%%
\section{Conclusion}\label{sec:final}

We have addressed the problem of discrete Bayesian network estimation from incomplete data by maximizing a penalized log-likelihood scoring function. The essential step in our approach is replacing the usual log-likelihood with a sum of node-average log-likelihoods, the so-called NAL. We have motivated our decision with a more efficient utilization of the available data and have shown the connection between NAL optimization and EM algorithm. Although our setup allows the missing data distribution to be arbitrary as long as the true DAG structure remains identifiable, the latter rarely holds for general MAR models. As we have demonstrated however, in MCAR settings, the identifiability of the set of independence relations, which characterizes all networks equivalent to the true one, is always guaranteed. 
We have shown, Theorems \ref{th:consistency}, that in presence of missing values the NAL-based estimator \eqref{eq:G_hat} requires more stringent conditions on the penalization parameter $\lambda_n$ to achieve consistency than in the complete data case. The discrepancy is due to the fact that in NAL each node may utilize different data subset for estimation thus reducing the overall convergence rate. Although the theorem guarantees consistency for penalties in a continuous range, choosing an optimal penalization parameter that performs well in finite sample settings is an open problem deserving further investigation.

The scope of this article has been limited to discrete BNs for which self-contained proofs of the results have been derived. It is straightforward however to apply NAL-based estimation to other classes of parametric BNs, such as linear Gaussian networks. Then, as long as for any $G$, $G_0\subset G$, $l(G|D_n)-l(G_0|D_n)$ is  $O_p(n^{-1/2})$, in the missing, and $O_p(n^{-1})$, in the complete data case, Theorem \ref{th:consistency}, with some technical modification of the proofs, seems to remain valid. Formulating identifiability and consistency for available case analysis in more general graphical model settings is thus a subject of continuing interest. 

%%%%%%%%%%%%%%%%%%%%%%%%%%%%%%%%%%%%%%%%%%%%%%%%%%%%%%%%%%%%%%%%%%%%%%%%%%%%%%%%%%%%%%%%%%%%%%%%%%%%%%%%%%%%%%%%%%%

\section{Proofs}
\label{appendixA}

The next lemma is instrumental in the proof of Proposition \ref{th:nal_increase}.
It shows that the (population) node log-likelihood $l(X|A)$, $X\in\{X_i\}_{i=1}^N$, $A\subset \{X_i\}_{i=1}^N$, is an increasing function of $A$ with respect to the set inclusion operation. In complete data settings, this result is better know as non-negativity of the Kullback-Leibler divergence. 

\begin{lem}
\label{lemma:lincrease}
For any $A,B \subset {\bf X}$ and $X\in {\bf X}$ such that $Z_B$ is independent of $(X, A, B)$ given $(Z=1,Z_A=1)$, we have $l(X|A) \le l(X|A,B)$. The inequality is strict if $P(X|A,Z=1,Z_A=1) \ne P(X|A,B,Z=1,Z_A=1)$.
\end{lem}

\begin{proof}[Proof]
Let
$$
\theta_{ka} \equiv P(X=k|A=a, Z=1, Z_A=1).
$$
By assumption
$$
\theta_{kab} \equiv P(X=k|A=a, B=b, Z=1, Z_A=1, Z_B=1)
$$
$$
= P(X=k|A=a, B=b, Z=1, Z_A=1)
$$ 
We can therefore write the expression 
$
\theta_{ka} = \sum_{b\in B} P(B=b|A=a, Z=1, Z_A=1) \theta_{kab}. 
$
By the convexity of the function $t\mapsto t\log(t)$ we have
$$
\theta_{ka}\log(\theta_{ka}) 
\le \sum_{b\in B} P(B=b|A=a,Z=1,Z_A=1) \theta_{kab} \log(\theta_{kab}), 
$$
and the claim follows from 
$$
l(X|A) = \sum_{a\in A} P(A=a|Z=1,Z_A=1) \sum_{k\in X} \theta_{ka}\log(\theta_{ka}) 
$$
$$
\le \sum_{a\in A, b\in B} P(B=b|A=a,Z=1,Z_A=1)P(A=a|Z=1,Z_A=1) \sum_{k\in X} \theta_{kab} \log(\theta_{kab}) 
$$
$$
= \sum_{a\in A, b\in B} P(A=a, B=b|Z=1,Z_A=1,Z_B=1) \sum_{k\in X} \theta_{kab} \log(\theta_{kab}) 
 = l(X|A,B).
$$
The last inequality is strict if $P(X|A,Z=1,Z_A=1) \ne P(X|A,B,Z=1,Z_A=1)$. 
\end{proof}

%%%%%%%%%%%%%%%%%%%%%%%%%%%%%%%%%%%%%%%%%%%%%%%%%%%%%%%%%%%%%%%%%%%%%%%%%%%%%%%%%%%%%%%%%%%%%%%%%%%%%%%%%%%%%%%%%%%

\begin{proof}[{\bf Proof of Proposition \ref{th:nal_increase}}]

{\raggedleft \bf Part (i)} \newline

Let $G_0\subset G$. 
The MCAR condition on ${\bf Z}$ and LMP imply that for every $i$ and $Y\subset {\bf X}$ such that $Y\prec_{G_0} X_i$ and $Y\cap Pa_i^0 = \emptyset$, the following two conditions hold 
\begin{enumerate}
\item[(i)]{$(Y,Z_Y)$ is independent of $X_i$ given $(Pa_i^0,Z_i=1,Z_{Pa_i^0}=1)$. 
}
\item[(ii)]{$Z_Y$ is independent of $Pa_i^0$ given $(Z_i=1,Z_{Pa_i^0}=1)$. 
}
\end{enumerate}
For each $i$, since $Pa_i^0\subset Pa_i$, $Pa_i\backslash Pa_i^0 \prec X_i$, by (i) we have that $(Pa_i\backslash Pa_i^0, Z_{Pa_i\backslash Pa_i^0})$ and $X_i$ are independent conditionally on $Pa_i^0$, 
and therefore for each $j\in Pa_i^0$ and $j'\in Pa_i\backslash Pa_i^0$, 
$$
\theta_{i,k(jj')}(G|G_0) = 
P(X_i=k|Pa_i=(jj'),Z_i=1,Z_{Pa_i}=1) 
$$
$$
 = P(X_i=k|Pa_i^0=j, Z_i=1,Z_{Pa_i^0}=1) = \theta_{i,kj}^0.
$$ 
Moreover, by (ii) applied to $Z_{Pa_i\backslash Pa_i^0}$ and $Pa_i^0$
$$
\sum_{j'\in Pa_i\backslash Pa_i^0} \theta_{i,jj'} = P(Pa_i^0=j|Z_i=1, Z_{Pa_i}=1) 
$$
$$
= P(Pa_i^0=j|Z_i=1, Z_{Pa_i^0}=1) = \theta_{i,j}^0 , 
$$
which implies 
$$
l(X_i|Pa_i) = \sum_{j\in Pa_i^0}\sum_{j'\in Pa_i\backslash Pa_i^0} \theta_{i,jj'} \sum_{k\in X_i} \theta_{i,k(jj')} \log \theta_{i,k(jj')}  
$$
$$
= \sum_{j\in Pa_i^0} \sum_{k\in X_i} \theta_{i,k(j)} \log \theta_{i,k(j)}  
= l(X_i|Pa_i^0). 
$$
We thus have $l(G|G_0) = l(G_0)$. 

{\raggedleft \bf Part (ii) and Part (iii)} \newline
By the definition of $l(G|G_0)$ in \eqref{eq:loglikG_G0} and some summation manipulations we obtain 
$$
l(G|G_0) = \sum_{i=1}^N \sum_{x_{Pa_i}\in Pa_i} \sum_{x_i \in X_i} P_0(X_i = x_i, Pa_i = x_{Pa_i}) \log P_{0}(X_i=x_i | Pa_i = x_{Pa_i}) 
$$
$$
= \sum_{i=1}^N \sum_{x \in {\bf X}} P_0({\bf X}=x) \log P_{0}(X_i=x_i | Pa_i = x_{Pa_i}) 
= \sum_{x \in {\bf X}} P_0(x) \log P_{G|G_0}(x), 
$$
where $x$ indexes the states of ${\bf X}$. 
Since $\sum_{x} P_0(x) = 1$, $\sum_{x} P_{G|G_0}(x) = 1$ and the $\log$-function is concave, we have 
$$
l(G|G_0) = \sum_{x} P_0(x) \log P_{G|G_0}(x)  \le \sum_{x} P_0(x) \log P_0(x) = l(G_0),
$$
with equality that is achieved only when $P_{G|G_0} = P_0$. 

\end{proof}

%%%%%%%%%%%%%%%%%%%%%%%%%%%%%%%%%%%%%%%%%%%%%%%%%%%%%%%%%%%%%%%%%%%%%%%%%%%%%%%%%%%%%%%%%%%%%%%%%%%%%%%%%%%%%%%%%%%

\begin{proof}[{\bf Proof of Corollary \ref{cor:ident_inorder}}]

Let $G_0\in\mathcal{G}$ be DBN with a node order $\Omega$ and $\mathcal{G}$ be a set of DAGs compatible with $\Omega$. We need to show that $l(G|G_0)<l(G_0)$ for all $G\in\mathcal{G}$ for which $G_0\nsubseteq G$.

Note that for any $G\in\mathcal{G}$, $G\cup G_0$ is also a DAG compatible with $\Omega$ and by Proposition \ref{th:nal_increase}, 
$$
l(X_i|Pa_i) \le l(X_i|Pa_i\cup Pa_i^0) = l(X_i|Pa_i^0)
$$
for all $i$. Moreover, because $G_0\nsubseteq G$, there is an $i$ such that $Pa_i = (Pa_i^0\backslash Y)$ for $Y$, $\emptyset\ne Y \subset Pa_i^0$. Since $P(X_i|Pa_i, Z_i=1, Z_{Pa_i}=1) \ne P(X_i|Pa_i^0, Z_i=1, Z_{Pa_i^0}=1)$, because by definition $G_0$ is a minimal DAG compatible with $P_0$, by Lemma \ref{lemma:lincrease}, $l(X_i|Pa_i^0\backslash Y) < l(X_i|Pa_i^0)$, implying $l(G|G_0)<l(G_0)$.  
\end{proof}

%%%%%%%%%%%%%%%%%%%%%%%%%%%%%%%%%%%%%%%%%%%%%%%%%%%%%%%%%%%%%%%%%%%%%%%%%%%%%%%%%%%%%%%%%%%%%%%%%%%%%%%%%%%%%%%%%

The next result is used in the proof of Lemma \ref{th:l_unif_conv}. 

\begin{lem}
\label{lemma:phi_log_theta_conv}
Let $\hat\phi_n = \phi + O_p(n^{-1/2})$ and $\hat\theta_n = \theta + O_p(n^{-1/2})$ for $\theta > 0$. Then 
\begin{equation}
\hat\phi_n \log(\hat\theta_n) - \phi \log(\theta) = O_p(n^{-1/2}). 
\label{eq:phi_log_theta_conv}
\end{equation}
\end{lem}

\begin{proof}[Proof]
By applying Taylor expansion to the logarithm function, we can write 
$$
\hat\phi_n \log(\hat\theta_n) - \phi \log(\theta) = (\hat\phi_n-\phi) \log(\theta) + \frac{\phi_n}{\eta_n}(\hat\theta_n - \theta)), 
$$
for some $\eta_n$ between $\hat\theta_n$ and $\theta$. Since $\phi_n/\eta_n \to_p \phi/\theta < \infty$, the claims follows from the assumptions. 
\end{proof}

%%%%%%%%%%%%%%%%%%%%%%%%%%%%%%%%%%%%%%%%%%%%%%%%%%%%%%%%%%%%%%%%%%%%%%%%%%%%%%%%%%%%%%%%%%%%%%%%%%%%%%%%%%%%%%%%%

\begin{proof}[{\bf Proof of Lemma \ref{th:l_unif_conv}}] 
Let $G\in\mathcal{G}$ has parent sets $Pa_i$. 
For all $i=1,...,N$, $j\in Pa_i$ and $k\in X_i$, we define 
$$
\xi_{i,kj} \equiv \hat{\theta}_{i,j} \hat{\theta}_{i,kj} \log \hat{\theta}_{i,kj} - \theta_{i,j} \theta_{i,kj} \log \theta_{i,kj},
$$
where $\theta_{i,j} = \theta_{i,j}(G|G_0)$ and $\theta_{i,kj} = \theta_{i,kj}(G|G_0)$, and $\hat\theta_{i,j}$ and $\hat\theta_{i,kj}$ are the corresponding estimates. In this notation we have 
$$
l(X_i|Pa_i, D_n)-l(X_i|Pa_i, G_0) = \sum_{i=1}^N \sum_{j\in Pa_i} \sum_{k\in X_i} \xi_{i,kj}. 
$$

Note that when either $\theta_i = 0$, $\theta_{i,j} = 0$ or $\theta_{i,kj}=0$ holds, then the state $(i,kj)$ will be unobservable and $\xi_{i,kj} = 0$. We thus may assume without loss of generality that $\theta_i>0$, $\theta_{i,j}>0$ and $\theta_{i,kj} >0$ for all $i,j$ and $k$. Moreover, 
by Hoeffding's inequality, for all $k,j, \epsilon>0$, $P(|\hat\theta_{i,kj}-\theta_{i,kj}|\ge\epsilon)\le 2\exp(-2n_{i,j}\epsilon^2)$ and hence $\hat\theta_{i,j} = \theta_{i,j} + O_p(n^{-1/2})$ and $\hat\theta_{i,kj} = \theta_{i,kj} + O_p(n^{-1/2})$. We can therefore apply Lemma \ref{lemma:phi_log_theta_conv} to $\hat{\theta}_{i,j} \hat{\theta}_{i,kj}$ and $\hat{\theta}_{i,kj}$ to infer that $\xi_{i,kj} = O_p(n^{-1/2})$, from which the claim follows. 

\end{proof}

%%%%%%%%%%%%%%%%%%%%%%%%%%%%%%%%%%%%%%%%%%%%%%%%%%%%%%%%%%%%%%%%%%%%%%%%%%%%%%%%%%%%%%%%%%%%%%%%%%%%%%%%%%%%%%%%%
The following two lemmas are essential for the proof of Theorem \ref{th:consistency}. The first one extends Lemma \ref{lemma:phi_log_theta_conv}.

\begin{lem}
\label{lemma:theta_log_theta_conv}
Let for $m=1,...,k$, $\hat\theta_m = \theta_0 + O_p(n^{-1/2})$, $\theta_0 > 0$ and 
$\hat\theta = \sum_{m}\gamma_m\hat\theta_m$, for $\gamma_m \ge 0$ such that $\sum_{m}\gamma_m = 1$. 
Then 
\begin{equation}
\Delta \equiv \sum_{m} \gamma_m\hat\theta_m \log(\hat\theta_m) - \hat\theta\log(\hat\theta) = O_p(n^{-1}). 
\label{eq:theta_log_theta_conv}
\end{equation}
\end{lem}

\begin{proof}[Proof] Note that $\hat\theta^m$, $\hat\theta$ and $\gamma_m$ are all considered to be random variables. 
By applying Taylor expansion to the logarithm function, we can write 
$$
\hat\theta_m \log(\hat\theta_m) = \hat\theta_m [ \log(\hat\theta) + \frac{1}{\hat\theta}(\hat\theta_m - \hat\theta) - \frac{1}{2\eta_m^2}(\hat\theta_m-\hat\theta)^2], 
$$
for some $\eta_m$ between $\hat\theta_m$ and $\hat\theta$. 
Since $\eta_m \to_p \theta_0 > 0 $, by assumption, we have $(\hat\theta_m-\hat\theta)^2/\eta_m^2 = O_p(n^{-1})$. After some algebra we obtain  
$$
\Delta = \sum_{m} \gamma_m [ \hat\theta_m \log(\hat\theta_m) - \hat\theta_m\log(\hat\theta) ]
= \frac{1}{\hat\theta} [\sum_m \gamma_m \hat\theta_m^2 - \hat\theta^2] + O_p(n^{-1}) 
$$
$$
 = \frac{1}{\hat\theta} [\sum_m \gamma_m (\hat\theta_m - \theta_0)^2 - (\hat\theta-\theta_0)^2] + O_p(n^{-1}). 
$$
Finally, since $1/\hat\theta \to_p 1/\theta_0 < \infty$ and $(\hat\theta-\theta_0)^2 = O_p(n^{-1})$, \eqref{eq:theta_log_theta_conv} follows. 
\end{proof}

%%%%%%%%%%%%%%%%%%%%%%%%%%%%%%%%%%%%%%%%%%%%%%%%%%%%%%%%%%%%%%%%%%%%%%%%%%%%%%%%%%%%%%%%%%%%%%%%%%%%%%%%%%%%%%%%%%
The next lemma presents a central limit result for difference between sample and sub-sample averages. It consequently establishes a variability rate of $n^{-1/2}$ for such differences. 

\begin{lem}
\label{lemma:draw_without_replace_clt} 
Let for each $n>0$, $x_n^1,...,x_n^n$ be i.i.d. random variables with mean $\mu_n$ and variance $\sigma_n^2$ such that $\sigma_n^2 \to \sigma^2 < \infty$. Let for some fixed $\beta\in(0,1)$, $k_n \sim Binom(\beta, n)$ and $\alpha_n$ be a random draw without replacements of $k_n$ elements from the set $\{1,...,n\}$. Then for 
$$
S_n = \frac{1}{k_n} \sum_{i \in \alpha_n} x_i - \frac{1}{n} \sum_{i=1}^n x_i, 
$$
for almost every sequence $\{\alpha_n\}_{n\ge 1}$, we have
\begin{equation}\label{eq:sn_convergence}
\sqrt{n}S_n \to_d \mathcal{N} (0, \frac{1-\beta}{\beta}\sigma^2).
\end{equation}
\end{lem}

\begin{proof}[Proof]
First, we rearrange the elements of $S_n$
$$
S_n = \frac{n-k_n}{nk_n}\sum_{i \in \alpha_n} x_i - \frac{1}{n}\sum_{i \notin \alpha_n} x_i.
$$
and define $\{y_n^j\}_{j=1}^{k_n}$ to be $\frac{n-k_n}{nk_n}x_{i}$ for $i\in \alpha_n$, and 
$\{z_n^j\}_{j=1}^{n-k_n}$ to be $-\frac{1}{n}x_i$ for $i\notin\alpha_n$. 
Hence $S_n = y_n^1+...+y_n^{k_n} + z_n^1+...+z_n^{n-k_n}$.

We are going to apply the Lindeberg-Feller CLT (see for example Th. 2.27 in \cite{vaart}) to the triangular sequence 
$$
\sqrt{n}y_n^1,...,\sqrt{n}y_n^{k_n},\sqrt{n}z_n^1,...,\sqrt{n}z_n^{n-k_n}.
$$ 
A key observation is that for each $n$, since $\alpha_n$ is a random draw without replacements, conditionally on $\alpha_n$, $y_n$'s and $z_n$'s are independent. Moreover, it is easy to verify that  
$$
E(S_n|k_n) = \frac{1}{n}(\frac{n-k_n}{k_n} k_n\mu - (n-k_n)\mu) = 0, 
$$
and
$$
Var(\sqrt{n}S_n|k_n) = \sum_{i=1}^{k_n} nVar(y_n^i) + \sum_{i=1}^{n-k_n} nVar(z_n^i) = \frac{1}{n}(\frac{(n-k_n)^2}{k_n} + (n-k_n))\sigma_n^2.
$$
By the law of large numbers, $k_n/n \to \beta$ almost surely, and therefore  
$$
Var(\sqrt{n}S_n|k_n) \to_{a.s.} \frac{1-\beta}{\beta}\sigma^2.
$$
It is left to verify that for any $\epsilon > 0$
$$
\sum_{i=1}^{k_n} nE(y_n^i)^21_{\{\sqrt{n}|y_n^i| > \epsilon\}} + \sum_{i=1}^{n-k_n} nE(z_n^i)^21_{\{\sqrt{n}|z_n^i| > \epsilon\}} \to 0, 
$$
almost surely for $\{\alpha_n\}_{n\ge 1}$. 
Indeed, the left-hand side sum equals
$$
\frac{(n-k_n)^2}{nk_n} E[(x_n)^21_{\{\frac{n-k_n}{k_n}|x_n|>\sqrt{n}\epsilon\}}]
+
\frac{n-k_n}{n} E[(x_n)^21_{\{|x_n|>\sqrt{n}\epsilon\}}]
$$
and, by the assumptions, the above expectations converge to $0$ for almost every sequence $\{k_n\}_{n\ge 1}$, and hence, for almost every $\{\alpha_n\}_{n\ge 1}$. 
Therefore the Lindeberg-Feller CLT is applicable and \eqref{eq:sn_convergence} holds. 

\end{proof}

%%%%%%%%%%%%%%%%%%%%%%%%%%%%%%%%%%%%%%%%%%%%%%%%%%%%%%%%%%%%%%%%%%%%%%%%%%%%%%%%%%%%%%%%%%%%%%%%%%%%%%%%%%%%%%%%%%%

\begin{proof}[{\bf Proof of Theorem \ref{th:consistency}}]
We shall first outline some key steps in the proof. The following notion will be useful: 
for a given sample $D_n$ and two DAGs $G_1=\{Pa_i^1\}$ and $G_2=\{Pa_i^2\}$, we say that the NALs $l(G_1|D_n)$ and $l(G_2|D_n)$ are estimated upon one and the same sample, if for every $i$, $l(X_i|Pa_i^1)$ and $l(X_i|Pa_i^2)$ are estimated from one and the same subsample of $D_n$, that is, for every $x^s\in D_n$, $(X_i,Pa_i^1)$ and $(X_i,Pa_i^2)$ are either both observed or both unobserved (missing) in $x^s$. When $(X_i,Pa_i)$ is observed in all $x^s\in D_n$, we say that $D_n$ is complete with respect to $(X_i,Pa_i)$. 

{\it
The essential problem of achieving consistent estimation is to decide between the true model $G_0$ and a more complex model $G_1$ in which $G_0$ is nested. According to Lemma \ref{th:l_unif_conv}, the NAL scores $l(G_1|D_n)$ and $l(G_0|D_n)$ both converge to $l(G_0)$ at a rate of $n^{-1/2}$ and so does their difference - this is essentially Lemma \ref{lemma:phi_log_theta_conv}. If therefore the scoring function penalty $\lambda_n$ converges to 0 at a slower than $n^{-1/2}$ rate, in the limit, $G_0$ would be preferred to $G_1$ as a less complex model. The latter condition is sufficient for both complete and missing data (claim (i) of the theorem). 
However, it turns out that when $l(G_1|D_n)$ and $l(G_0|D_n)$ are estimated upon one and the same sample, as in the complete data case, their difference converges at a faster rate of $n^{-1}$ - this is essentially due to the result of Lemma \ref{lemma:theta_log_theta_conv}. Then we can relax the necessary convergence rate of $\lambda_n$ and still achieve consistency (claim (ii)). The application of Lemma \ref{lemma:theta_log_theta_conv} crucially depends on the condition: for every node $X_i$, if a sample is complete with respect to $(X_i,Pa_i^0)$ so it is with respect to $(X_i,Pa_i^1)$. In MCAR settings the latter does not hold and the difference $l(G_1|D_n) - l(G_0|D_n)$ has a persistent variability of order $n^{-1/2}$, due to a central limit result, Lemma \ref{lemma:draw_without_replace_clt}. Consequently the condition on $\lambda_n$ to diminish at a rate slower than $n^{-1/2}$ becomes both sufficient and necessary (claim (iii)). 
}

The more formal proof follows. We shall prove consistency by verifying conditions (C1) and (C2) in Proposition \ref{th:mle_cons}. We first assume that  $G_1,G_2\in\mathcal{G}$ are such that $G_0\subseteq G_1$ and $G_0 \nsubseteq G_2$. 
Then, by the identifiability of $G_0$, Definition \ref{def:ident}, we have $l(G_2|G_0) < l(G_1|G_0)$. 
Moreover, by Lemma \ref{th:l_unif_conv}, regardless of the observation probability $\beta(\mathcal{G}) > 0$, $l(G_1|D_n)\to_p l(G_1|G_0)$ and $l(G_2|D_n)\to_p l(G_2|G_0)$, and hence, for $\delta = (l(G_1|G_0)-l(G_2|G_0))/2$, $P(l(G_1|D_n) > l(G_2|D_n) + \delta)\to 1$, as $n\to\infty$. The consistency condition (C1) therefore holds because the sequence $\lambda_n$ diminishes with $n$. 

As in the proof of Lemma \ref{th:l_unif_conv}, without loss of generality we may assume that for all $G\in\mathcal{G}$, $i=1,...,N$, $j\in Pa_i(G)$ and $k\in X_i$, $\theta_{i,j}(G|G_0) > 0$ and $\theta_{i,kj}(G|G_0) > 0$. 

%%%%%%%%%%%%%%%%%%%%%%%%%%%%%%%%%%%%%%%%%%%%%%%%%%%%%%%%%%%%%%%%%%%%%%%%%%%%%%%%%%%%%%%%%%%%%%%%%%%%%%%%%%%%%

{\raggedleft \bf Part (i)} \newline
Let now assume $G_0\subseteq G_1$, $G_0 \subseteq G_2$ and $h(G_1) < h(G_2)$. 
To verify the consistency condition (C2) we need to find the rate of convergence of the random variable $l(G_1|D_n) - l(G_2|D_n)$. This rate depends on whether the data is complete or not. 

By Lemma \ref{th:l_unif_conv}, $l(G|D_n) = l(G|G_0) + O_p(n^{-1/2})$ and since $l(G_1|G_0) = l(G_2|G_0)$, we have
$$
l(G_1|D_n) - l(G_2|D_n) = O_p(n^{-1/2}).
$$
The latter holds regardless of the observation probability $\beta(\mathcal{G}) > 0$. 
Therefore, the condition $\sqrt{n}\lambda_n \to \infty$ implies that the positive sequence $\lambda_n(h(G_2) - h(G_1))$ will overcome the likelihood difference $l(G_2|D_n) - l(G_1|D_n)$ with probability approaching 1, that is, 
$
P(S(G_1|D_n) > S(G_2|D_n)) \to 1 \textrm{, as } n\to\infty. 
$
This proves the first part of the theorem. \newline

%%%%%%%%%%%%%%%%%%%%%%%%%%%%%%%%%%%%%%%%%%%%%%%%%%%%%%%%%%%%%%%%%%%%%%%%%%%%%%%%%%%%%%%%%%%%%%%%%%%%%%%%%%%%%

{\raggedleft \bf Part (ii)} \newline
In case of complete data, $\beta(\mathcal{G})=1$, we shall obtain a faster convergence rate of $n^{-1}$ for the difference $l(G|D_n) - l(G_0|D_n)$, $G_0\subseteq G$, which will prove the second part $(ii)$ of the claim. 

We consider the sample average log-likelihood of the node $X_i$. Since $G_0\subseteq G$, we have that $Pa_i = Pa_i^0 \cup Y$ for some $Y\subset \{X_i\}_{i=1}^N$, $Y\cap Pa_i^0=\emptyset$. Observe that in $G_0$, $X_i$ cannot have descendants in $Y$. Indeed, if there is a directed path $X_i$ to $X_s\in Y$ in $G_0$, this path cannot be in $G$ also, for otherwise one would have the loop $X_i$ to $X_s$ to $X_i$ in $G$ and $G$ would not be a DAG. But if there is a path in $G_0$ that is not in $G$, then $G_0\nsubseteq G$, a contradiction.
Therefore, by LMP, $X_i$ and $Y$ are independent given $Pa_i^0$. 

In the usual notation, for $j\in Pa_i^0$, $m\in Y$ and $k\in X_i$, 
$\hat\theta_{i,kjm}$ denotes the estimator of $P(X_i=k|Pa_i^0 = j, Y=m)$ and  $\hat\theta_{i,jm}$ is the estimator of $P(Pa_i^0 = j, Y=m)$. 
We start with the expression 
$$
l(X_i|Pa_i, D_n) - l(X_i|Pa_i^0, D_n) 
$$
$$
= \sum_{m\in Y} \sum_{j\in Pa_i^0} \sum_{k\in X_i} \hat{\theta}_{i,jm} \hat{\theta}_{i,kjm}\log(\hat{\theta}_{i,kjm}) 
- \sum_{j\in Pa_i^0} \sum_{k\in X_i}  \hat{\theta}_{i,j} \hat{\theta}_{i,kj}\log(\hat{\theta}_{i,kj}) 
$$
\begin{equation}\label{eq:difflik_paext}
= \sum_{j\in Pa_i^0} \sum_{k\in X_i} \hat{\theta}_{i,j} \Delta_{i,kj},
\end{equation}
where 
$$
\Delta_{i,kj} \equiv 
\sum_{m\in Y} \frac{\hat{\theta}_{i,jm}}{\hat{\theta}_{i,j}} \hat{\theta}_{i,kjm}\log(\hat{\theta}_{i,kjm}) - \hat{\theta}_{i,kj}\log(\hat{\theta}_{i,kj}).
$$
By definition
$$
\hat{\theta}_{i,jm} = \frac{n_{i,jm}}{\sum_{j',m'} n_{i,j'm'}} \textrm{, }
\hat{\theta}_{i,j} = \frac{n_{i,j}}{\sum_{j'} n_{i,j'}} \textrm{ and }
\hat{\theta}_{i,kjm} = \frac{n_{i,kjm}}{n_{i,jm}}.
$$
By the sample completeness, we have $\sum_{m}n_{i,jm} = n_{i,j}$ and $\sum_{m}n_{i,kjm} = n_{i,kj}$, implying 
\begin{equation}\label{eq:gammasum1}
\sum_{m\in Y} \frac{\hat{\theta}_{i,jm}}{\hat{\theta}_{i,j}} \hat{\theta}_{i,kjm} 
= \frac{ \sum_{j'} n_{i,j'} }{ \sum_{j',m} n_{i,j'm} } \frac{\sum_{m} n_{i,kjm} }{n_{i,j}}  = \frac{n_{i,kj}}{n_{i,j}} = \hat{\theta}_{i,kj}.
\end{equation}
If we set $\gamma_m = \hat{\theta}_{i,jm}/\hat{\theta}_{i,j}$, then $\sum_m\gamma_m = 1$ and $\sum_m \gamma_m \hat\theta_{i,kjm} = \hat\theta_{i,kj}$. However, the latter are not guaranteed in incomplete settings because then we may have $\sum_{m}n_{i,jm} < n_{i,j}$ and(or) $\sum_{m}n_{i,kjm} < n_{i,kj}$. 

We can now apply Lemma \ref{lemma:theta_log_theta_conv} with $\gamma_m$, $\hat{\theta}_{i,kjm}$ and $\hat{\theta}_{i,kj}$ to infer that $\Delta_{i,kj} = O_p(n^{-1})$ and 
\begin{equation}\label{eq:lrate_plogp_beta1}
l(X_i|Pa_i, D_n) - l(X_i|Pa_i^0, D_n) = O_p(n^{-1}).
\end{equation}
Therefore $l(G|D_n) - l(G_0|D_n) = O_p(n^{-1})$ holds for all $G$ such that $G_0\subseteq G$. 

If both $G_1$ and $G_2$ contain $G_0$, it follows that $l(G_2|D_n) - l(G_1|D_n) = O_p(n^{-1})$. 
Therefore, the condition $n\lambda_n \to \infty$ implies that the positive sequence $\lambda_n(h(G_2) - h(G_1))$ will overcome the likelihood difference $l(G_2|D_n) - l(G_1|D_n)$ with probability approaching 1, which concludes the second part $(ii)$ of the theorem. \newline

\begin{rem}
\label{rem:on_convergence} 
\eqref{eq:lrate_plogp_beta1} holds even for incomplete samples $D_n$ if they satisfy the property: $(X_i,Pa_i)$ is complete in $D_n$ whenever $(X_i,Pa_i^0)$ is complete, or equivalently, $l(X_i|Pa_i, D_n)$ and $l(X_i|Pa_i^0, D_n)$ are calculated upon one and the same subsample of $D_n$. Lemma \ref{lemma:theta_log_theta_conv} is applicable in such cases because we still have  $\sum_{m}n_{i,jm} = n_{i,j}$ and $\sum_{m}n_{i,kjm} = n_{i,kj}$, and consequently, $\sum_m \gamma_m = 1$ and $\sum_m \gamma_m  \hat\theta_{i,kjm} = \hat\theta_{i,kj}$. Interestingly, if $Z$ is MCAR with $P(Z_Y|Z_i=1,Z_{Pa_i^0}=1)=a\in(0,1)$, then ${\sum_{m}n_{i,jm}}/{n_{i,j}} \to_{p} a$ and ${\sum_{m}n_{i,kjm}}/{n_{i,kj}} \to_{p} a$, and consequently $\sum_m \gamma_{m} \to_{p} 1$ and $\sum_m \gamma_m  \hat\theta_{i,kjm} - \hat\theta_{i,kj} \to_{p} 0$; this however is not enough for the claim in Lemma \ref{lemma:theta_log_theta_conv} to hold. 
\end{rem}

%%%%%%%%%%%%%%%%%%%%%%%%%%%%%%%%%%%%%%%%%%%%%%%%%%%%%%%%%%%%%%%%%%%%%%%%%%%%%%%%%%%%%%%%%%%%%%%%%%%%%%%%%%%%%
%% (iii)

{\raggedleft \bf Part (iii)} \newline
The last part of the theorem claims the necessity of condition $(i)$ in case of incomplete sample that also satisfies Condition \ref{eq:assumption_incons}. Let $i$ be an unique node index of $G\in\mathcal{G}$ for which the condition holds, that is, for $j\ne i$, $Pa_j=Pa_j^0$, but $Pa_i\backslash Pa_i^0\ne\emptyset$. Without loss of generality we may assume that $h(G)=h(G_0)+1$ and that $D_n$ is complete with respect to $(X_i, Pa_i^0)$. 
We shall show that for $\lambda_n$ such that $\underline{\lim}\sqrt{n}\lambda_n < \infty$
\begin{equation}\label{eq:deltai_incons}
\overline{\lim}_{n\to\infty} P(l(X_i|Pa_i,D_n) - l(X_i|Pa_i^0,D_n) - \lambda_n > 0) > 0,
\end{equation}
which is equivalent to $S$ to be inconsistent.

Let $\tilde D_n=\{\tilde{x}^t\}_{t=1}^{\tilde n}$ be the $\tilde n$-subsample of $D_n$ for which $(X_i, Pa_i)$ is observed. Then we have $l(X_i|Pa_i,D_n) = l(X_i|Pa_i,\tilde D_n)$. 
Note that $\tilde n$ is random and $\tilde n \sim Binom(a, n)$, 
for $a = P(Z_{Pa_i}=1|Z_i=1,Z_{Pa_i^0}=1)\in(0,1)$, by Condition \ref{eq:assumption_incons}. 

For every probability table $\theta$, we denote 
$$
l(X_i|Pa_i^0, \theta, \tilde D_n) \equiv \frac{1}{\tilde n} \sum_{t=1}^{\tilde n} l(\tilde x^{t}|Pa_i^0, \theta),
$$
and
$$
l(X_i|Pa_i^0, \theta, D_n) \equiv \frac{1}{n} \sum_{s=1}^{n} l(x^s|Pa_i^0, \theta)
$$
where 
$$
l(x|Pa_i^0, \theta) \equiv \sum_{j\in Pa_i^0}\sum_{k\in X_i} 1_{x_i=k, pa_i^{0}=j} \log(\theta_{i,kj}).
$$
We have 
$$
l(X_i|Pa_i^0,D_n) = l(X_i|Pa_i^0,\hat\theta_n,D_n) \textrm{, for } \hat\theta_n = arg\max_{\theta} l(X_i|Pa_i^0,\theta,D_n)
$$
and 
$$
l(X_i|Pa_i^0,\tilde D_n) = l(X_i|Pa_i^0,\tilde\theta_n,\tilde D_n) \textrm{, for } \tilde\theta_n = arg\max_{\theta} l(X_i|Pa_i^0,\theta,\tilde D_n).
$$

Next, we show that the convergence rate of the difference  $l(X_i|Pa_i^0,\hat\theta_n,D_n)-l(X_i|Pa_i^0,\theta_0,D_n)$ is $n^{-1}$. The function $f(\theta)\equiv l(X_i|Pa_i^0,\theta,D_n)$ has continuous first and second derivatives in a neighborhood of $\theta_0$. 
Since $\frac{\partial f}{\partial\theta}|_{\hat\theta_n} = 0$ ($f$ has a maximum at $\hat\theta_n$), the Taylor's expansion of $f$ at $\theta=\hat\theta_n$ is 
$$
f(\theta_0) = f(\hat\theta_n) + 0.5(\theta_0-\hat\theta_n)^T \frac{\partial^2 f}{\partial\theta\partial\theta^T}|_{\theta_n^*} (\theta_0-\hat\theta_n) \textrm{, for } \theta_n^*\in[\theta_0,\hat\theta_n]. 
$$
Moreover, the Hessian at $\theta_0$ is bounded because $\theta_0$ is bounded away from $0$ and $\theta_n^*\to_p\theta_0$. Hence 
$$
\frac{\partial^2 f}{\partial\theta\partial\theta^T}|_{\theta_n^*} = \frac{\partial^2 f}{\partial\theta\partial\theta^T}|_{\theta_0} + O_p(1). 
$$
Because $||\theta_0-\hat\theta_n||^2 = O_p(n^{-1})$, we infer 
$$
\sqrt{n}(l(X_i|Pa_i^0,\hat\theta_n,D_n) - l(X_i|Pa_i^0,\theta_0,D_n)) = O_p(n^{-1/2}). 
$$
Similarly we have 
\begin{equation}
\label{eq:delta_theta_rate}
\sqrt{n}(l(X_i|Pa_i^0,\tilde\theta_n,\tilde D_n) - l(X_i|Pa_i^0,\theta_0, \tilde D_n)) = O_p(n^{-1/2}). 
\end{equation}

Due to the MCAR assumption, $Z_{Pa_i\backslash Pa_i^0}$ is independent of $(X_i,Pa_i^0)$. 
This and Condition \ref{eq:assumption_incons} imply that $\tilde D_n$ is obtained from $D_n$ by random draws without replacements. Therefore, we can apply Lemma \ref{lemma:draw_without_replace_clt} to the set $\{l(x^s|Pa_i^0,\theta_0)\}_{s=1}^n$ of i.i.d. random variables and 
its subset $\{l(\tilde x^t|Pa_i^0,\theta_0)\}_{t=1}^{\tilde n}$. We thus infer 
\begin{equation}\label{eq:ldiff_normal}
\sqrt{n} (l(X_i|Pa_i^0, \theta_0, \tilde D_n) - l(X_i|Pa_i^0, \theta_0, D_n)) \to_d \mathcal{N}(0, \gamma Var(l(X_i|Pa_i^0,\theta_0))),
\end{equation}
where $\gamma = (1-a)/a>0$. Also note that $Var(l(X_i|Pa_i^0,\theta_0)) > 0$ by the identifiability of $G_0$. 

Moreover, by Remark \ref{rem:on_convergence}, \eqref{eq:lrate_plogp_beta1} applied to $\tilde D_n$ yields 
\begin{equation}\label{eq:lrate_theta}
\sqrt{n}(l(X_i|Pa_i, \tilde D_n) - l(X_i|Pa_i^0, \tilde D_n)) = O_p(n^{-{1/2}}). 
\end{equation}

Finally, we consider the difference implicated in \eqref{eq:deltai_incons} 
$$
\sqrt{n}(l(X_i|Pa_i, \tilde D_n) - l(X_i|Pa_i^0, D_n)) - \sqrt{n}\lambda_n 
$$
apply \eqref{eq:lrate_theta}
$$
=  \sqrt{n}(l(X_i|Pa_i^0, \tilde\theta_n, \tilde D_n) - l(X_i|Pa_i^0, \hat\theta_n, D_n)) - \sqrt{n}\lambda_n + O_p(n^{-1/2})
$$
then use \eqref{eq:delta_theta_rate} 
$$
= \sqrt{n}(l(X_i|Pa_i^0, \theta_0, \tilde D_n) - l(X_i|Pa_i^0, \theta_0, D_n)) - \sqrt{n}\lambda_n + O_p(n^{-1/2}) =: T_n + O_p(n^{-1/2}). 
$$
Since $\underline{\lim}\sqrt{n}\lambda_n < \infty$, there is a subsequence $n'$ such that $\lim\sqrt{n'}\lambda_{n'} = \lambda_0 < \infty$ and for which, taking into account the convergence \eqref{eq:ldiff_normal}, we have 
$$
T_{n'} \to_d \mathcal{N}(-\lambda_0, \gamma Var(l(X_i|Pa_i^0,\theta_0))).
$$ 
Therefore $\lim P(T_{n'} > 0) \to 1-\Phi(\lambda_0/\sqrt{\gamma Var(l(X|Pa_i^0,\theta_0))}) > 0$, where $\Phi$ is the c.d.f. of the standard normal distribution.  
Hence \eqref{eq:deltai_incons} is verified and with this the proof of the theorem. 

\end{proof}

\section*{Acknowledgements}

This work was supported by NIH grant K99LM009477 from the National Library
Of Medicine. The content is solely the responsibility of the author and
does not necessarily represent the official views of the National Library
Of Medicine or the National Institutes of Health.
The author thanks an anonymous reviewer whose comments greatly improved the paper and who suggested the 2-node simulated example in Section \ref{ch:experiments}, and also Peter Salzman for many stimulating discussions.

\end{document}